\documentclass[11pt]{amsart}

\usepackage[pdftex]{graphicx}
\usepackage{subfig}
\usepackage{amssymb}
\usepackage{graphbox}
\usepackage{amsmath}
\usepackage{amsfonts}
\usepackage{amsthm}
\usepackage{amssymb}
\usepackage{mathrsfs}
\usepackage{xcolor}
\usepackage{url}
\numberwithin{equation}{subsection}
\newtheorem{theorem}{Theorem}[section]

\newtheorem{proposition}[theorem]{Proposition}
\newtheorem{corollary}[theorem]{Corollary}

\newtheorem{lemma}[theorem]{Lemma}

\theoremstyle{definition}
\newtheorem{definition}[theorem]{Definition}
\theoremstyle{remark}
\newtheorem{remark}[theorem]{Remark}

\newcommand{\bd}{\partial}

\newcommand{\mL}{\mathcal L}
\newcommand{\B}{\mathcal B}
\newcommand{\A}{\mathcal A}
\newcommand{\mM}{\mathcal M}

\newcommand{\ca}{\mathrm{Cap}}
\newcommand{\supp}{\mathrm{supp}}

\DeclareMathOperator{\area}{\mathrm{Area}}
\DeclareMathOperator{\ind}{\mathrm{ind}}

\begin{document}

\title[Conformally maximal metrics for eigenvalues on surfaces]{Conformally maximal metrics for  Laplace eigenvalues on surfaces}
\begin{abstract} 

The paper is concerned with the maximization of Laplace eigenvalues on surfaces of given volume with a Riemannian metric in a fixed conformal class.  
A significant  progress on this problem  has been  recently achieved by Nadirashvili--Sire and Petrides using related, though different methods.  
In particular, it was  shown that for a given $k$, the maximum of the $k$-th Laplace eigenvalue in a conformal class on a surface is either attained  on a metric which is smooth except possibly at a finite number of  conical  singularities,  or it is attained in the limit while a ``bubble tree''  is formed on a surface. Geometrically, the bubble tree appearing in this setting can be viewed as a union  of touching identical round spheres. We present another proof of this statement, developing the approach proposed by the second author and Y. Sire. As a side result, we provide explicit upper bounds on the topological spectrum of surfaces.


\end{abstract}
\author[M. Karpukhin]{Mikhail Karpukhin}
\address[Mikhail Karpukhin]{Department of Mathematics, University of California,  Irvine, 340 Rowland
Hall, Irvine, CA 92697-3875.}
\email{mkarpukh@uci.edu}
\author[N. Nadirashvili]{Nikolai Nadirashvili}
\address[Nikolai Nadirashvili]{CNRS, I2M UMR 7353 --- 
Centre de Math\'e\-ma\-ti\-ques et Informatique,
Marseille, France}
\email{nikolay.nadirashvili@univ-amu.fr}
\author[A. Penskoi]{Alexei V. Penskoi\textsuperscript{1}}
\address[Alexei V. Penskoi]{Department of Higher Geometry and Topology, 
Faculty of Mathematics and Mechanics, Moscow State University,
Leninskie Gory, GSP-1, 119991, Moscow, Russia \newline {\em and}
\newline Faculty of Mathematics,
National Research University Higher School of Economics,
6 Usacheva Str., 119048, Moscow, Russia \newline {\em and}
\newline Interdisciplinary Scientific Center
J.-V. Poncelet (ISCP, UMI 2615), Bolshoy Vlasyevskiy 
Pereulok 11, 119002, Moscow, Russia} 
\email{penskoi@mccme.ru}
\author[I. Polterovich]{Iosif Polterovich\textsuperscript{2}}
\address[Iosif Polterovich]{D\'epartement de math\'ematiques et de 
statistique, Universit\'e de Montr\'eal, CP 6128 Succ. Centre-Ville, 
Montr\'eal, Qu\'ebec, H3C 3J7, Canada}
\email{iossif@dms.umontreal.ca}
\date{}
\footnotetext[1]{Partially supported by the Simons-IUM fellowship.}
\footnotetext[2]{Partially supported by 
NSERC.}
\maketitle
\section{Introduction and main results}
\subsection{Conformally maximal metrics}  Let $M$ be a compact surface without boundary  endowed with a Riemannian metric $g$. 
The corresponding measure is denoted by $dv_g$,
and in what follows all integrations and functional spaces are considered with respect to this measure unless indicated otherwise.
 Let $\Delta_g$ be the associated Laplace-Beltrami operator on $M$ with eigenvalues $$0=\lambda_0<\lambda_1 \le \lambda_2 \le \dots \lambda_n \dots \nearrow +\infty$$
and corresponding eigenfunctions $\phi_n$, forming an orthonormal basis in $L^2(M)$.

Given a conformal class $\mathcal{C}$ of Riemannian metrics on $M$, let
\begin{equation}
\label{def:max}
\Lambda_k(M,\mathcal{C})=\sup_{g\in \mathcal{C}} \lambda_k(M,g) \area(M,g).
\end{equation}
It is well-known that this supremum is always finite. In fact, for surfaces, the supremum is finite even if taken over all Riemannian metrics, not necessarily conformally equivalent, see  \eqref{eq:Kor}. This was conjectured by S.-T. Yau 
(\cite[Problem 71]{Yau}, reprinted in  \cite[Chapter VII]{SY})   and later proved in \cite{Korevaar}, see also \cite{Gro, GY, GNY, Hass} for further  developments. For details we refer to  subsection \ref{subsec:Korevaar},  where explicit estimates of this kind are obtained.

The numbers  $\Lambda_k(M,\mathcal{C})$ are called the {\it conformal spectrum} of $M$ (see \cite{CES}).
The goal of the present paper is to provide a new proof of the following result due to Nadirashvili--Sire and Petrides (see  \cite{NS1, NS2} and  \cite{Pet1,  Pet2}):
\begin{theorem}
\label{thm:main} Let $M$ be a compact Riemannian surface without  boundary.  

\noindent (i)  For any conformal class $\mathcal{C}$ of Riemannian metrics on $M$,  there exists a metric $g~\in~\mathcal{C}$, possibly with a finite number of conical singularities, such that 
\begin{equation}
\label{firsteig}
\Lambda_1(M,\mathcal{C})=\lambda_1(M,g) \area(M,g). 
\end{equation}

\smallskip

\noindent (ii) For  any conformal class $\mathcal{C}$ of Riemannian metrics on $M$ and for any $k>1$, either one has
\begin{equation}
\label{mbubbles}
\Lambda_k(M,\mathcal{C}) =\Lambda_{k-1}(M,\mathcal{C})+8\pi,
\end{equation}
or there exists a metric $g~\in~\mathcal{C}$, possibly with a finite number of conical singularities, such that 
\begin{equation}
\label{eq:mainineq}
\Lambda_k(M,\mathcal{C})=\lambda_k(M,g) \area(M,g)>  \Lambda_{k-1}(M,\mathcal{C}) + 8\pi.
\end{equation}
%
%
%
\end{theorem}
\begin{remark}
\label{rem:thmmain}
Applying  part (ii) of Theorem \ref{thm:main} iteratively, we arrive at an alternative that could be informally stated as follows.  Given $k>1$, either there exists a maximal metric for $\lambda_k$ which is smooth outside a  finite number of conical singularities, or the supremum $\Lambda_k(M,\mathcal C)$  is achieved by  a sequence of metrics degenerating (in a sense to be specifed in Section \ref{sec:atoms}) to a disjoint union of $1<j<k$ identical round spheres of volume $8\pi/\Lambda_k(M,\mathcal C)$ (so-called ``bubbles'', see subsection \ref{subsec:bubbles}) and  the surface $M$ endowed with a maximal metric 
for $\Lambda_{k-j}(M,\mathcal C)$, which is smooth away from a finite number of conical singularities.  Note that the number of conical singularities of a maximal metric for $\Lambda_k(M,\mathcal{C})$ is bounded above in terms of $k$ and the genus of $M$, see \cite[Proposition 1.13]{Kar3}.
 Let us also mention that  equality \eqref{mbubbles} can be stated in such a simple form since  
\begin{equation}
\label{eq:KNPP}
\Lambda_k(\mathbb{S}^2):=\Lambda_k(\mathbb{S}^2, [g_{\text{\rm st}}]) =8\pi k  
\end{equation}
for all $k\ge 1$, as was shown in \cite{KNPP}. 
Here $g_{\text{\rm st}}$ is the standard round metric, and we recall that any Riemannian metric on $\mathbb{S}^2$ is conformally equivalent  to $g_{\text{\rm st}}$.
Note that apart from the specific value of the constant $8\pi$ in~\eqref{mbubbles} and~\eqref{eq:mainineq},  the proof  of Theorem~\ref{thm:main} does not rely on~\eqref{eq:KNPP}. While the proof of~\eqref{eq:KNPP} uses the dichotomy in Theorem~\ref{thm:main} (ii),  it  does not require to specify the  value of the constant. Hence, there is no 
  ``circular argument'' in the    proofs of Theorem~\ref{thm:main} and formula~\eqref{eq:KNPP}. 

   We note as well that  if one replaces the strict inequality  in \eqref{eq:mainineq} by a non-strict inequality, it would be always true by the standard gluing argument
(see \cite[Theorem B]{CES} and \cite[Remark~2.4]{KNPP}).
\end{remark}
\begin{remark}
Part (i) of Theorem \ref{thm:main}
 appeared first in \cite{NS1}  under the assumption
\begin{equation}
\label{assum:sphere}
\Lambda_1(M,\mathcal{C})>8\pi=\Lambda_1(\mathbb{S}^2). 
\end{equation}
It  was later  shown in \cite{Pet1} that this inequality holds for any conformal class $\mathcal{C}$ on any surface $M$ which is not a sphere. 
\end{remark}
\begin{remark}  The exact values of $\Lambda_k(M,\mathcal{C})$ and the corresponding maximizing metrics are known in a very few cases. 
Apart from the result  \eqref{eq:KNPP} for the sphere (see  \cite{Nad3, Pet3, NS3, KNPP}) and the equality
\begin{equation}
\label{eq:projective}
\Lambda_k(\mathbb{RP}^2):=\Lambda_k(\mathbb{RP}^2, [g_ {\text{\rm st}}]) =4\pi (2k+1), \,\,\, k\ge 1  
\end{equation}
for the real projective plane \cite{Kar2} (see also \cite{LY, NaPe}),  nothing is known in the case $k>1$.
 For $k=1$, global maximizers (i.e. maximizers over all all conformal classes)  have been found for the sphere \cite{Hersch}, the real projective plane \cite{LY}, the torus \cite{Nad}, the Klein bottle \cite{JNP, EGJ, CKM} and the surface of genus two \cite{JLNNP, NaSh}. It is also known that for certain conformal classes on tori,  the first eigenvalue is maximized by the Euclidean metric \cite{EIR}. Finally, let us note that for $k=1$ the analogue of part (i) of Theorem \ref{thm:main} for global maximizers  has been recently proved in \cite{MS}.
\end{remark}
\begin{remark}
It is mentioned  in~\cite{Pet2} that inequality~\eqref{eq:mainineq} 
holds for some conformal classes. Indeed,  Let $g_0$ be a smooth metric on the genus $2$ surface $\Sigma_2$ obtained by gluing  together two copies of the equilateral torus $\mathbb{T}_{eq}^2$ 
using a short thin tube.  Since $2 \lambda_1(\mathbb{T}_{eq}^2)\area(\mathbb{T}_{eq}^2)=\dfrac{16\pi^2}{\sqrt{3}}>24\pi$, this gluing can be done so that
$$
\lambda_2(\Sigma_2,g_0) \area(\Sigma_2, g_0)>24\pi.
$$
At the same time, 
$\Lambda_1(\Sigma_2, [g_0]) \le 16\pi$ by the Yang-Yau inequality (\cite{YY}, see also \cite{EI, JLNNP, NaSh}).
Therefore, one has
$$
\Lambda_2(\Sigma_2,[g_0])>24\pi\geqslant \Lambda_1(\Sigma_2,[g_0]) + 8\pi.
$$
Under certain assumptions on the gluing procedure (in particular, if the radius of the connecting tube is small compared to its length, and if the  metric $g_0$ has nonpositive curvature everywhere),  it follows from \cite[Theorem 2.4]{BE} that the conformal class $[g_0]$  is close to the boundary 
of the moduli space of $\Sigma_2$. For such conformal classes, 
inequality~\eqref{eq:mainineq} could be also deduced as follows. Consider a degenerating sequence $\mathcal C_n$ of conformal classes of the genus $2$ surface $\Sigma_2$, converging (on the Deligne-Mumford compactification of the moduli space) to two copies of the equilateral torus. Using the continuity result~\cite[Theorem 2.8]{KM} one has that 
$$
\lim\limits_{n\to\infty} \Lambda_2(\Sigma_2,\mathcal C_n) \geqslant  2\Lambda_1(\mathbb{T}^2,\mathcal C_{eq}) = \frac{16\pi^2}{\sqrt{3}}.
$$ 
Hence,  as before, for large enough $n$ one has
$$
\Lambda_2(\Sigma_2,\mathcal C_n)>24\pi\geqslant \Lambda_1(\Sigma_2,\mathcal C_n) + 8\pi.
$$
\end{remark}
\subsection{Explicit upper bounds on the topological spectrum}
\label{subsec:Korevaar}
Given a surface $M$, set 
$$
\Lambda_k(M)=\sup_{\mathcal{C}} \Lambda_k(M,\mathcal{C}).
$$
The numbers $\Lambda_k(M)$ are sometimes called the {\it topological spectrum} of  the surface $M$. The explicit values of $\Lambda_k(M)$ are known only in a few cases,  see \cite[Section 2]{KNPP} for an overview. If $M$ is a  surface of genus $\gamma$, it was shown by Korevaar \cite[Theorem 0.5]{Korevaar} (see also \cite{GNY, Hass}) that
\begin{equation}
\label{eq:Kor}
\Lambda_k(M) \le C (\gamma+1),
\end{equation}
where $C$ is some universal constant.  Though this bound has been formally stated for orientable surfaces, its proof  works for non-orientable surfaces as well.
Combining the ideas  of \cite{YY, Kar0} with the estimates \eqref{eq:KNPP}  and \eqref{eq:projective}, we can make the constant $C$ explicit.
\begin{theorem}
\label{thm: Korevaar}
\noindent {\rm (i)}  Let $M$ be an orientable surface of genus $\gamma$. Then
\begin{equation}
\label{eq:KorexplO}
\Lambda_k(M) \leqslant 8\pi k \left[\frac{\gamma+3}{2}\right], \,\,\, k\geqslant 1.
\end{equation}

\noindent {\rm  (ii)}   Let $M$ be an non-orientable surface, and let $\gamma$ be the genus of its orientable double cover. 
Then
\begin{equation}
\label{eq:KorexplN}
\Lambda_k (M) \leqslant 16\pi k  \left[\frac{\gamma+3}{2}\right], \,\,\, k\geqslant 1.
\end{equation}
\end{theorem}

The proof is presented in subsection~\ref{sec:Korevaar}.

\subsection{Plan of the  proof of Theorem \ref{thm:main}}
The methods  used  in \cite{NS1, NS2} and \cite{Pet1, Pet2} to prove Theorem \ref{thm:main}  are different, though they share some common tools.  
The approach developed in \cite{Pet1, Pet2}  uses the heat equation techniques in an essential way.  The argument outlined in \cite{NS1, NS2} uses a reformulation of the eigenvalue optimisation problem in terms of Schr\"odinger operators (see also \cite{GNS}). In this  paper we  present a proof of Theorem \ref{thm:main} developing the approach of  \cite{NS1, NS2}. We clarify some of the ideas that were put forward   in those papers,  and introduce
several  new ingredients which are needed to complete the argument. 

Let us describe the main  parts of the proof of Theorem \ref{thm:main}.
The first part  essentially follows \cite{NS1, GNS}.
We start by fixing a metric $g\in \mathcal{C}$ on $M$ of constant curvature satisfying $\area_g(M) = 1$, and 
use the  conformal invariance of $\Delta$ to reduce our consideration to a family of eigenvalue problems
\begin{equation}
\label{EVequation}
\Delta u = \lambda V u,
\end{equation}
where $V\in L^1(M)$ is a positive function with the unit $L^1$-norm. Geometrically, the potential $V$  (under an additional assumption $V\in C^\infty(M)$)  represents the conformal factor for a metric $g'=Vg \in \mathcal{C}$,
and the condition $||V||_1=1$ means that $\area_{g'}(M)=1$.\textsuperscript{3}
\footnotetext[3]{Slightly abusing notation, in what follows we identify metrics with their corresponding conformal factors.}
This leads to an optimisation problem ($\mathcal{A}_k$) defined in subsection \eqref{subsec:opt}. 
For the reasons explained below, we would like to consider the eigenvalue equation~\eqref{EVequation} for not necessarily positive functions $V$. 
 However, in that case the corresponding spectral problem is not elliptic, since 
the quadratic form $Q(u) = \int_M Vu^2\,dv_g$ is no longer  positive definite.
In order to circumvent this difficulty we reformulate the problem~\eqref{EVequation} in terms of a certain Schr\"odinger operator, see subsection~\ref{subsec:schrod}. Using this reformulation we  introduce an optimisation problem 
($\mathcal{B}_k$) which is in a sense equivalent to ($\mathcal{A}_k$). At the same time, it  admits simpler extremality conditions, because it allows more general perturbations. This leads to Theorem \ref{thm:summary} which states that for each $k \ge 1$, there exists a maximizing sequence of (possibly singular) metrics of area one defined by the potentials $V_{N_m,k}$, satisfying $||V_{N_m,k}||_{L^\infty}\leqslant C{N_m}$ for some constant $C$, 
such that the corresponding eigenfunctions $\phi_{N_m,k}$ converge weakly in $H^1$ as $N_m \to \infty$. This brings us to the next step of the argument, because the weak convergence of eigenfunctions  is not enough to deduce the required regularity properties of the limiting metric.

The second part  of the proof is described in Section \ref{regularity:sec}. We define the {\it good points} (see Definition \ref{def:good}),  which are characterised by having a neighborhood with a sufficiently large first Dirichlet eigenvalue. In a way, this means that the measure does not concentrate too much near a good point; if this condition is violated, we say that a point is {\it bad}.
Using variational arguments we show that all but possibly $k$ points on $M$ are good. The key technical result of  Section \ref{regularity:sec}  is Proposition \ref{regularity} which shows that $\phi_{N_m,k}$ converge strongly in $H^1$ in a neighborhood of a good point. We refer to Remark  \ref{rem:epsreg} for an interpretation of this result as an $\varepsilon$-regularity type theorem (see \cite{ColdMin}).  
The proof of Proposition \ref{regularity} requires rather delicate auxiliary analytic results which are proved in Section \ref{preliminaries:sec} using the theory of capacities. Some of them, such as 
Lemma \ref{lemma_H0} could be of independent interest.  Using Proposition \ref{regularity} we prove Theorem \ref{regularity:thm}, which is the main result of this section.  It states that the limiting measure is regular away 
from bad points, while the latter give rise to $\delta$-measures, i.e. atoms.
The regular part is constructed via a harmonic map defined using limiting eigenfunctions. The harmonic map theory (see \cite{Helein, Kok2}) then yields that the regular part may have at most a finite number of conical singularities.  Sections 3 and 4 contain probably the most novel ingredients of the proof of Theorem \ref{thm:main}.  

The last part of the proof is presented in Section \ref{sec:atoms}. It describes the behaviour  of a  maximizing sequence of metrics near the atoms, and  in a sense is a variation of the bubble tree construction for harmonic maps \cite{Parker}. A similar construction using somewhat different analytic tools could be also found in \cite{Pet2}. This  part of the proof is technically quite involved and could be subdivided into several steps.
We choose a normalization parameter $C_R$ that eventually will tend to zero, and rescale the metric near the bubble. The constant $C_R$ determines the rescaling, and controls the size of the bubbles: in particular, we will ignore bubbles of size less than $C_R$ (note that  for each fixed $k$, bubbles of sufficiently small size do not affect $\lambda_k$).  In the process of rescaling,  secondary bubbles (i.e. the descendants of the initial bubble on the bubble tree) may arise.  However, we show in Lemma \ref{finiteness:lemma}   that each time a secondary bubble appears, its size  decreases in a controlled way and therefore the bubble tree is finite. Moreover, $C_R$ controls the size of the necks, i.e. the areas  between a bubble and its descendant on the bubble tree. 
 
Using the rescaling and an inverse stereographic projection to the sphere, we view each bubble $b$ as a sphere with  a sequence of metrics with conformal factors $\widetilde{V}^b_{N_m,k}$  which are obtained from the maximizing sequence $V_{N_m,k}$ described above. Away from the secondary bubbles, this sequence converges weakly to a limiting metric defined by a potential $\widetilde{V}^b_\infty \in L^1(\mathbb{S}^2)$, see Theorem \ref{tree:summaryI}.
 If this potential is bounded, we say that the bubble is of type I,
otherwise we say that  the bubble $b$ is of type II. In the latter case the Laplacian on $\mathbb{S}^2$ with the limiting metric defined by $\widetilde{V}^b_\infty$ may have essential spectrum, see Remark \ref{rem:essential}. For type II bubbles the construction of the test-functions for the eigenvalue $\lambda_k$  is significantly more involved than for type I  bubbles, see subsection \ref{subsec:test}.  In section \ref{subsec:eigbounds}  we estimate the Rayleigh quotients of test functions separately on differents parts of the surface $M$: the smooth part, the type I bubbles, the type II bubbles and necks. Taking $C_R \to 0$ and applying \eqref{assum:sphere} together with  \cite[Theorem 1.2]{KNPP},  we complete the proof of Theorem \ref{thm:main}.
\subsection*{Acknowledgments} The authors would like to thank Alexandre Girouard for valuable remarks on the earlier version of the paper. 
We are also thankful to Dorin Bucur and Daniel Stern for useful  discussions, as well as to Leonid Polterovich for pointing out  the reference \cite{BE} and helpful comments.
\section{Two optimisation problems}
\subsection{Optimisation of the eigenvalues in a conformal class}
\label{subsec:opt}
Consider the spectral problem~\eqref{EVequation} with a nonnegative $V\in L^1(M)$. It can be understood in the weak form, where $V\,dv_g$ is treated as a Radon probability measure (see \cite{Kok2}). The eigenvalues of such a problem can be characterised variationally via Rayleigh quotient, i.e one defines
\begin{equation}
\label{Variational1}
\lambda_k(V) = \inf_{E_k}\sup_{u\in E_k}\frac{\displaystyle\int_M|\nabla u|^2\,dv_g}{\displaystyle\int_M V\,u^2\,dv_g},
\end{equation}
where the supremum is taken over $E_k \subset C^\infty(M)$ which form  $(k+1)$-dimensional subspaces  in $L^2(V\,dv_g)$. The latter condition is equivalent to saying that the restriction of $E_k$ to $\mathrm{supp}\,V$ is $(k+1)$-dimensional. Note that we enumerate the eigenvalues starting from $\lambda_0(V)$.

\begin{remark}  Alternatively, variational characterisation \eqref{Variational1} could be written in the form
\begin{equation}
\lambda_k(V) = \sup_{F_k}\inf_{u\perp F_k}\frac{\displaystyle\int_M|\nabla u|^2\,dv_g}{\displaystyle\int_MVu^2\,dv_g},
\end{equation}
where $F_k$ is now $k$-dimensional and $\perp$ is understood in $L^2(V\,dv_g)$. 
\end{remark}
Let  $\mathcal{N}(\lambda) = \#\{\lambda_i(V)<\lambda\}$  be the eigenvalue counting function. The following proposition holds.
\begin{proposition}
\label{index1}
Let $Q^V_\lambda(u)$ be the following quadratic form
$$
Q^V_\lambda(u) = \int_M|\nabla u|^2\,dv_g - \lambda\int_M u^2\,Vdv_g.
$$
Then $\mathcal{N}(\lambda)=\ind Q^V_\lambda$, i.e. the right-hand side is defined as the maximal dimension of a linear subspace on which
$Q^V_\lambda$ is negative definite.
\end{proposition}
\begin{proof} Let us show first that  $\ind Q^V_\lambda\leqslant N(\lambda)$.
Suppose there exists a $k$-dimensional space $G_k$, where $Q^V_\lambda$ is negative definite. We will prove  that $\lambda_{k-1}(V)<\lambda$, which implies $\mathcal{N}(\lambda)\geqslant k$. The first observation is that $G_k$ remains $k$-dimensional in $L^2(Vdv_g)$. Indeed, if $u\in G_k$ is such that $\int_M u^2\,Vdv_g=0$, then $Q^V_\lambda(u,u)\geqslant 0$.
Thus we can use $G_k$ in the variational characterisation~\eqref{Variational1} for $\lambda_{k-1}(V)$. The claim then follows, since for any $u\in G_k$ one has 
$$
\int_M |\nabla u|^2\,dv_g<\lambda\int_M uv\,Vdv_g.
$$
Let us now prove the inequality in the opposite direction: $\ind Q^V_\lambda\geqslant N(\lambda)$.
Let $k = \mathcal{N}(\lambda)-1$ and let $G_{k+1}$ be the $k+1$-dimensional space spanned by the first $k$ eigenfunctions and the constants (corresponding to $\lambda_0=0$).  Then it is easy to see that $Q^V_\lambda$ is negative definite on $G_k$. This implies that $\ind Q^V_\lambda=\mathcal{N}(\lambda)$ and completes the proof of the proposition.
\end{proof}

Consider the class of functions
\begin{equation*}
\mL_N = \{V|\, V\in L^\infty(M),\,\, 0 \leqslant V\leqslant N\, ,||V||_1 = 1\}
\end{equation*}
endowed with *-weak topology coming from identity $(L^1)^* = L^\infty$. Since $\mL_N$ is bounded  a subset of $L^\infty$, it is compact in this topology by Banach-Alaoglu theorem. Moreover, one has the following Proposition.  
\begin{proposition}\cite[Proposition 1.1]{Kok2}
\label{usc}
Functional $\lambda_k(V)$ is upper semicontinuous on $\mL_N$.
\end{proposition}
Consider the following optimisation problem:
\begin{equation}
\label{Problem1}
\lambda_k(V) \to \max\quad\mathrm{for}\quad V\in \mL_N.
\tag{$\mathcal{A}_k$}
\end{equation}
The following proposition holds:
\begin{proposition}{\rm (\cite{Korevaar}, see also \cite{GNY}).}
\label{Korevaar-proposition}
There exists a universal constant $C$ such that for all $V\in\mL_N$ one has
$$
\lambda_k(V)\leqslant Ck.
$$
\end{proposition}
Propositions~\ref{usc} and ~\ref{Korevaar-proposition} combined imply that there exists a solution to problem~\eqref{Problem1}. Set 
$$\tilde \Lambda_k^N=\max_{V \in \mathcal{L}_N} \lambda_k(V).$$
Since any non-negative $L^\infty$ function can be approximated by positive $C^\infty$ function such that the Laplacian spectra converge, see~\cite[Lemma 4.5]{CKM}, one has
\begin{equation}
\label{lknc}
\lim_{N\to \infty} \tilde \Lambda_k^N=\Lambda_k(M,\mathcal C).
\end{equation}
\subsection{Negative eigenvalues of Schr\"odinger operator}
\label{subsec:schrod}
Let $\Delta - W$ be the classical Schr\"odinger operator with $W\in L^\infty(M)$. Let $\sigma_k(W)$ denote the corresponding eigenvalues, i.e. real numbers $\sigma$ such that there exist a non-zero solution of
\begin{equation}
\label{SchrodingerEV}
\Delta u - Wu = \sigma u.
\end{equation}

The eigenvalues $\sigma_k(W)$ admit a variational characterisation as follows,
\begin{equation}
\label{Variational2}
\sigma_k(W) = \inf_{E_k}\sup_{u\in E_k}\frac{\displaystyle\int_M\left(|\nabla u|^2 - Wu^2\right)\,dv_g}{\displaystyle\int_M u^2\,dv_g},
\end{equation}
where $E_k$ ranges over $(k+1)$-dimensional subspaces in $C^\infty(M)$.
\begin{remark} 
\label{rem:k+1}
We enumerate eigenvalues $\sigma_k$ starting from $k=0$, similarly to the eigenvalues $\lambda_k$. Thus, $\sigma_k$ is in fact the $(k+1)$-st eigenvalue of the 
Schr\"odinger operator $\Delta-W$.
\end{remark}
The following proposition is proved similarly to Proposition \ref{index1}:
\begin{proposition}
\label{index2}
Let $\mathcal{N}_- = \#\{\sigma_i<0\}$. Then $\mathcal{N}_- = \ind Q^W_1$.
\end{proposition}

 Define the class of functions
\begin{equation*}
\mM_{N,k} = \{W\in L^\infty(M),\,\,   \sigma_k(W)\geqslant 0,\,\,  ||W||_\infty\leqslant N\}.
\end{equation*}
It was shown in~\cite[Lemma 2.1]{GNS} that $\mM_{N,k}$ is compact in *-weak topology.

The second optimisation problem is the following,
\begin{equation}
\label{Problem2}
\int_M W\,dv_g \to \max \quad\mathrm{for}\quad W\in \mM_{N,k}.
\tag{$\mathcal{B}_k$}
\end{equation} 
This functional is obviously bounded by $N$, therefore there exist a solution to this problem.

The following theorem essentially states that problems~\eqref{Problem1} and~\eqref{Problem2} are equivalent in the limit $N\to\infty$.

\begin{theorem}
\label{thmWnk}
Let $\{W_{N,k}\}_N$ be a sequence of solutions to \eqref{Problem2} for $k\geqslant 1$. Set
$$
V_{N,k} = \frac{W_{N,k}}{\int_M W_{N,k}\,dv_g}.
$$ 
Then for sufficiently large $N$ one has
\begin{itemize}
\item[(i)] $V_{N,k}\geqslant 0$ almost everywhere;
\item[(ii)] $\lambda_k(V_{N,k}) = \int_MW_{N,k}\,dv_g$;
\item[(iii)] $\limsup\limits_{N\to+\infty}\lambda_k(V_{N,k}) = \Lambda_k(M,c)$.
\end{itemize}
\end{theorem}

Let us first explain the significance of this result. As we will see below the extremality condition for the problem~\eqref{Problem2} is more tractable than the corresponding condition for~\eqref{Problem1}. It is a consequence of the fact that the space $\mM_{N,k}$ contains functions which may take negative values, i.e. there is more freedom in choosing a perturbation.

%
%
%
%

\begin{proof}
The main part of the proof of this theorem is contained  in~\cite{GNS}. In particular, in~\cite[Lemma 3.1]{GNS} it is shown that for large enough $N$, the solution $W_{N,k}$ is non-negative almost everywhere (see also Remark \ref{Lemma3.3}), and in \cite[Lemma 2.2]{GNS} it is proved that $\sigma_k(W_{N,k}) = 0$.

\smallskip

\noindent {\it Proof of (i)}. Since $W_{N,k}\geqslant 0$ almost everywhere, it is sufficient to show that $W_{N,k}\not\equiv 0$ as an element of $L^\infty(M)$. But then for each $0<c<\lambda_1(M,g)$ the constant function $c\in\mM_{N,k}$ and therefore $W_{N,k}\equiv 0$ can not be a solution to~\eqref{Problem2}.

\smallskip

\noindent {\it Proof of (ii)}.  Note that for $\lambda = \int_MW_{N,k}\,dv_g$ the eigenvalue equation $(\Delta - \lambda  V_{N,k})u = 0$ coincides with the equation for the zero eigenvalue of the operator $\Delta - W_{N,k}$. In particular,  the multiplicity of $\lambda$ as an eigenvalue of the first equation equals the multiplicity  of  $0$ as the eigenvalue of  the second. Moreover, by Propositions~\ref{index1} and~\ref{index2}, one has $N(\lambda) = N_-$. As a result, we have that $\lambda$-eigenvalues of the first equation have the same indices as $0$-eigenvalues of the second equation. Since $\sigma_k(W_{N,k}) = 0$, then $\lambda_k(V_{N,k}) = \lambda$. 

\smallskip

\noindent {\it Proof of (iii)}.
Let $\tilde V_{N,k}\in \mathcal{L}_N$ be a solution to~\eqref{Problem1}: $\lambda_k(\tilde V_{N,k})=\Lambda_k^N$. Set $\tilde W_{N,k} = \lambda_k(\tilde V_{N,k})\tilde V_{N,k}$. Then $\tilde W_{N,k}\in \mM_{\tilde N,k}$ for 
$\tilde N = \lambda_k(\tilde V_{N,k})\,N+1$. 
Then by (ii) we have
\begin{equation}
\label{ltk}
\lambda_k(V_{\tilde N,k}) = \int_M W_{ \tilde N,k}\,dv_g \geqslant \int_M \tilde W_{N,k}\,dv_g = \lambda_k(\tilde V_{N,k})=\tilde \Lambda_k^N,
\end{equation}
where the inequality in the middle follows from the fact that $W_{\tilde N,k}$ is a solution of \eqref{Problem2} in $\mM_{\tilde N,k}$.
Note that the right-hand side of \eqref{ltk} converges to $\Lambda_k(M,\mathcal C)$ by \eqref{lknc}, which implies $\limsup\limits_{N\to+\infty}\lambda_k(V_{N,k}) = \Lambda_k(M,\mathcal C)$.
\end{proof}

\begin{remark}
\label{Lemma3.3}
Let us remark that the statement of~\cite[ Lemma 3.3]{GNS}  that is used in the proof \cite[Lemma 3.1]{GNS} contains a minor inaccuracy. It requires the solution $v$ to be $C^2$, whereas in the sequel  Lemma 3.3 is applied to eigenfunctions of a Schr\"odinger operator with $L^\infty$ potential,  which are H\"older continuous (see, for instance, \cite{Ko12}), but 
not necessarily $C^2$. However, this is not a problem, since  Lemma 3.3 holds for a wider class of solutions. In particular, the proof presented in~\cite{GNS} remains valid under the assumption that $v\in H^1\cap C^0$ is a weak solution of
\cite[inequality (7)]{GNS}.
\end{remark}

%
%

%
%
%
%
%
%
%

In the following we use $V_{N,k}$ as a maximizing sequence for $\Lambda_k:=\Lambda_k(M,\mathcal{C})$ and assume $N$ is large enough so that Theorem \ref{thmWnk} holds. 
We also set 
\begin{equation}
\label{lambdakn}
\Lambda_k^N:=\lambda_k(V_{N,k}) = \int_M W_{N,k}\,dv_g.
\end{equation}

\subsection{Extremality conditions for problem~\eqref{Problem2}}
The reason we chose the maximizing sequence in this way is that the extremality condition for problem~\eqref{Problem2} has a particularly convenient form which we derive below.
We will use the following well-known lemma (see \cite[Lemma 3.2]{GNS}, see also \cite[Theorem 2.6, section 8.2.3]{Kato}).
\begin{lemma} 
\label{perturb}
Let $W(t,x)$ be a function on $\mathbb R\times M$ such that for any $t\in\mathbb{R}$, $W(t,\cdot)\in L^\infty(M)$ and $\bd_t W(t,\cdot)\in L^\infty(M)$. For any $t\in\mathbb{R}$ consider the eigenvalue problem
$$
\Delta u - W(t,\cdot)u = \sigma(t)u
$$
Denote by $\{\sigma_l(t)\}$ the sequence of the eigenvalues counted with multiplicity and arranged in increasing order. Suppose that 
$$
\sigma_{l-1}(0)< \sigma = \sigma_l(0) = \ldots = \sigma_{l+m-1}(0)<\sigma_{l+m}(0).
$$
Let $U_\sigma$ be the eigenspace for $t=0$ corresponding to $\sigma$. Define a bilinear form $Q$ on $U_\sigma$ by 
\begin{equation}
\label{cond1}
Q(u,v) = -\int_M\partial_tW(0,\cdot)uv\,dv_g
\end{equation}
and denote by $\alpha_i$, $i=0,\ldots,m-1$ the eigenvalues of this form with respect to the $L^2$ inner product,  counted with multiplicity and arranged in an increasing order. Then for any $i = 0,\ldots,m-1$, we have
\begin{equation}
\label{perturbDerivative}
\sigma_{l+i}(t) = \sigma_{l+i}(0) + t\alpha_i + o(t).
\end{equation}
\end{lemma}

For an element of the maximising sequence $V_{N,k}$  let $U_{N,k}$ to be the eigenspace corresponding to $\Lambda^N_k$. Note that $U_{N,k}$ is also 0-eigenspace for the problem~\eqref{SchrodingerEV} with $W = W_{N,k}:= \Lambda^N_kV_{N,k}$. By definition $0\leqslant W_{N,k}\leqslant N$. Set  
\begin{equation}
\label{Enk}
E_{N,k} = \left\{x\in M|\,\, W_{N,k}(x) \geqslant \frac{N}{2}\right\}.
\end{equation}

\begin{proposition}
\label{prop:aux1}
For any $v\in L^\infty(M)$ such that $\int_Mv = 0$ and $v\leqslant 0$ on $E_{N,k}$, there exists $u\in U_{N,k}\backslash\{0\}$ such that $\int_Mvu^2\geqslant 0$.
\end{proposition}
\begin{proof}
Assume the contrary, i.e. there exists $v\in L^\infty$ such that $\int\limits_M v = 0$; $v\leqslant 0$ on $E_{N,k}$ and for all $u\in U_{N,k}\backslash\{0\}$ one has
\begin{equation}
\label{cond2}
\int\limits_M vu^2\,dv_g<0.
\end{equation}

Set $W(t) = W_{N,k} + tv$. We first remark that  in view of \eqref{cond2}, the quadratic form \eqref{cond1} is positive definite, and therefore by \eqref{perturbDerivative} one has 
\begin{equation}
\label{contr}
\sigma_k(W(t))>\sigma_k(W(0)) = 0
\end{equation}
for $1\gg t>0$.  

Furthermore, we claim that $W(t)\in \mM_{N,k}$ for $1\gg t>0$.
Indeed, on $E_{N,k}$ one has $v\leqslant 0$ and $\frac{N}{2}\leqslant W_{N,k} \leqslant N$. Therefore,
$$
-N\leqslant \frac{N}{2} + tv\leqslant W_{N,k} + tv \leqslant N
$$
for $1\gg t >0$. At the same time, on $M\backslash E_{N,k}$ we have $0\leqslant W_{N,k}\leqslant \frac{N}{2}$, and hence
$$
-N\leqslant tv\leqslant W_{N,k} + tv\leqslant \frac{N}{2} + tv\leqslant N
$$
for $1\gg t >0$.

Thus, $W(t)\in \mM_{N,k}$ and $\int_MW(t) = \int_M W_{N,k}$, i.e. $W(t)$ is a solution to~\eqref{Problem2}. Recall that by \cite[Lemma 2.2]{GNS},  any solution to~\eqref{Problem2} has to satisfy $\sigma_k(W(t))  = 0$ and we arrive at a contradiction with \eqref{contr}.
\end{proof}
Proposition \ref{prop:aux1} allows us to obtain the following characterisation of solutions to~\eqref{Problem2}.

\begin{proposition}
\label{extremal}
There exists a collection $\phi_{N,k} = (u^1_{N,k},\ldots,u^{l(N)}_{N,k})$ of elements of $U_{N,k}$ such that $|\phi_{N,k}|^2 = \sum\limits_{i=1}^{l(N)} \left(u_{N,k}^i\right)^2= 1$ on 
$M\backslash E_{N,k}$,  $dv_g-a.e.$,  and  $w_{N,k} \leqslant 1$,  $dv_g-a.e.$
\end{proposition}
The proposition is an easy corollary of the lemma below.
\begin{lemma}
\label{HahnBanach}
Let $E\subset M$ be a measurable set in $M$. Let $Q$ be a convex finite-dimensional cone in $L^1(M)$ such that 
\begin{itemize}
\item[(i)] If $q\in Q$ then $q\geqslant 0$ a.e.
\item[(ii)] For any $v\in L^\infty(M)$ such that $\int_M v = 0$ and $v\leqslant 0$ a.e. on $E$ there exists $0\ne q\in Q$ such that $\int_M v q \geqslant 0$.
\end{itemize} 

Then there exists $q_0\in Q$ such that $q_0\equiv 1$ on $M\backslash E$ and $q_0\leqslant 1$ a.e.
\end{lemma}
\begin{proof}
The proof of this lemma is inspired by \cite[Theorem 5]{Nad}, \cite[Lemma 3.8]{NS1}.
Denote by $K$ the following convex cone
\begin{equation}
\label{coneK}
K = \{u\in L^1(M)|\, u\equiv 0\,\,\mathrm{on}\,\, M\backslash E,\, u\leqslant 0 \,\, \mathrm{a.e. \,\, on}\,\, E\}.
\end{equation}
First, we note that $K\cap Q = \{0\}$. Indeed, any $u\in K\cap Q$ satisfies $u\geqslant 0$ a.e. and $\int_M u\leqslant 0$ at the same time. Let $K_1$ be the convex cone spanned by $1$ and $K$. Suppose that $Q\cap K_1\ne\{0\}$. Then there exists $0\ne q\in Q$, $\alpha\geqslant 0$ and $k\in K$ such that $q = \alpha + k$. Since $K\cap Q = \{0\}$, one has that $\alpha\ne 0$. Therefore, $q_0 =\alpha^{-1}q$ satisfies the conditions of the lemma.

In the rest of the argument we assume the contrary, i.e. that $K_1\cap Q = \{0\}$. According to a Hahn-Banach type  result \cite[Theorem 2.7]{Klee},  there exists an element $v_0\in (L^1)^*(M) = L^\infty(M)$ such that for any $k_1\in K_1\backslash\{0\}$ and any $q\in Q\backslash\{0\}$ one has 
\begin{equation}
\label{eq:separating}
\int_M v_0k_1 >0>\int_M v_0q.
\end{equation} 
%

Set $v = v_0 - \int_M v_0$ so that 
\begin{equation}
\label{eq:intV}
\int_M v = 0. 
\end{equation}
Our goal is to show that $v$ contradicts property (ii) of the cone $Q$. 

First, note that for any $0\ne q\in Q$ one has 
\begin{equation}
\label{claim1}
\int_M qv <0.
\end{equation} 
Indeed, since $1\in K_1$, in view of the first inequality in \eqref{eq:separating} one has 
\begin{equation}
\label{eq:v0}
\int_M v_0 >0.
\end{equation} 
Therefore,
$$
\int_M qv = \int_M qv_0 - \int_M q\int_Mv_0 < 0.
$$ 
Note that the first term is negative by the second inequality in \eqref{eq:separating}, and both integrals in the second term are nonnegative due to \eqref{eq:v0} and  property (i) of the cone $Q$.


Second, let us show that for any $0\ne k\in K$ one has 
\begin{equation}
\label{claim2}
\int_M vk> 0. 
\end{equation}
Indeed, since $\int_M k \leqslant 0$  by \eqref{coneK},  one has $k - \int k\in K_1$. Therefore, by the first inequality in \eqref{eq:v0}  we have
$$
0<\int_M v_0\left(k - \int_M k \right) = \int_M v_0k - \int_M v_0\int_M k = \int_M vk.
$$
Now, take  $F = \{x\in E,\, v(x)>0\}$ and let $\chi_F$ be the characteristic function of $F$. Then $-\chi_F\in K$ and one has 
$$
0\leqslant -\int_M v\chi_F \leqslant 0,
$$
where the first inequality follows from \eqref{claim2} and the second inequality is trivial. Therefore, both  inequalities are equalities, which is possible iff $\area(F) = 0$.
Therefore, $v\leqslant 0$ a.e. on $E$. Together with \eqref{eq:intV} it means that $v$ satisfies the assumptions in property (ii) of the cone $Q$, and we get a contradiction with \eqref{claim1}.
This completes the proof of he lemma.
\end{proof}

\noindent{\em Proof of Proposition \ref{extremal}.}
In Lemma \ref{HahnBanach}, let $Q$ be the  convex hull of the squares of elements in $U_{N,k}$ and let $E:=E_{N,k}$. Note that property (ii) of $Q$ follows from Proposition \ref{prop:aux1} and property (i) is immediate.The result then follows by a direct application of Lemma \ref{HahnBanach}.
\qed

Proposition \ref{extremal} yields the following corollary.
\begin{corollary}
\label{cor:bounded}
There exists a constant $C$ such that for any $k,N\in \mathbb{N}$ and any $i=1,\dots, l(N)$,  we have $||u^i_{N_m,k}||^2_{H^1}\leqslant Ck$.
\end{corollary}
\begin{proof}  Indeed, it follows from Proposition \ref{extremal} that the $L^2$ norm  of $u^i_{N_m,k}$ is bounded above by a constant equal to  $\area(M)$. At the same time, the Dirichlet energy of 
$u^i_{N_m,k}$ is bounded by $Ck$ by Proposition \ref{Korevaar-proposition}. This completes the proof of the corollary.
\end{proof}

For future reference,  let us summarize the results of this section in the folowing theorem.
\begin{theorem}
\label{thm:summary}
For each $k\ge 1$, there exists a strictly increasing sequence $N_m$, $m=1,2,\dots$, of natural numbers, and maps $\phi_{N_m,k}=(u^1_{N_m,k}, \dots, u^d_{N_m,k}):M\to \mathbb{R}^d$ for some $d \in \mathbb{N}$, such that
\begin{itemize}
\item[(1)] $\Delta \phi_{N_m, k} = \Lambda_k^{N_m} \, V_{N_m, k}\, \phi_{N_m,k}$, \,\, $\Lambda_k^{N_m} \to \Lambda_k:=\Lambda_k(M,\mathcal{C})$.
\item[(2)] The $(k+1)$-st eigenvalue of the Schr\"odinger operator 
$$\Delta - \Lambda_k^{N_m}V_{N_m,k}$$  is zero.
\item[(3)] $||V_{N_m,k}||_{L^\infty}\leqslant C{N_m}$, $||V_{N_m,k}||_{L^1} = 1$.
\item[(4)] There exists a weak limit $\phi_{N_m, k} \rightharpoonup \phi_{k}= (u^1_{k}, \dots, u^d_{k})$
in $H^1$ and $\phi_{N_m, k} \to \phi_{k}$ in $L^2$.
\item[(5)] $|\phi_{N_m, k}|\leqslant 1$ and $|\phi_k| = 1$ $dv_g$-a.e.
\item[(6)] $V_{N_m,k}\, dv_g\rightharpoonup^*d\mu_k$ for some probability measure $d\mu_k$.
\end{itemize}
\end{theorem}
\begin{remark}
\label{rem:vectornot}
Here and in what follows, given a map $\phi=(u^1,\dots, u^d)$ we use the notation $|\phi|^2=\sum_{j=1}^d (u^j)^2$,
$|\nabla \phi|^2 =\sum_{j=1}^d |\nabla u^j|^2.$
\end{remark}
\begin{proof}
Note that by the multiplicity bounds of \cite{Ko12}, the dimension $l(N)$ in Proposition \ref{extremal} is bounded by a constant independent of $N$. Therefore, choosing an appropriate subsequence  we may assume that the images of the maps $\phi_{N_m,k}$ lie in $\mathbb{R}^d$ for some fixed $d$. In fact, below we will be extracting subsequences from $\phi_{N_m,k}$  on a number of occasions. Slightly abusing notation for the sake of simplicity, we will denote these subsequences again by $\phi_{N_m,k}$. 

Let us now prove assertions  (1--6). Properties (1--3) follow from Theorem \ref{thmWnk}. The weak convergence in $H^1$ of a subsequence $\phi_{N_m,k}$ follows from Corollary \ref{cor:bounded} and the fact that any bounded sequence in $H^1$ contains a weakly convergent subsequence. Since the embedding $H^1(M)\to L^2(M)$  is compact, one can extract a subsequence that strongly converges in $L^2$. This proves property (4). The first part of property (5) is a direct consequence of Proposition \ref{extremal}. In order to prove the second assertion of (5) we argue as follows. First, we note that the measures of the sets $E_{N,k}$ defined by \eqref{Enk} tend to zero as $N\to \infty$. 
Indeed, by \eqref{lambdakn} one has that
\begin{equation}
\label{Enk:bound}
\Lambda_k\geqslant \int_{E_{N,k}} W_{N,k}\,dv_g\geqslant \frac{N}{2}dv_g(E_{N,k}),
\end{equation}
i.e. $dv_g(E_{N,k})\leqslant CN^{-1}$.
Therefore, by Proposition \ref{extremal}, $|\phi_{N_m,k}|$ converges to $1$ in measure, and hence one may choose a subsequence   that converges almost everywhere to $1$. Therefore, by dominated convergence, $|\phi_{N_m,k}|$ converges to $1$ in $L^2$, and by property (4) we get that $|\phi_k|=1$ almost everywhere.  Finally, property (6) follows from the second part of (3),  since any bounded sequence of measures contains a *-weakly convergent subsequence. 
\end{proof}


\section{Analytic tools}
\label{preliminaries:sec}
\subsection{Capacity and quasi-continuous representatives.} Throughout this section let  $\Omega \subset \mathbb{R}^2$ be an open set.
Recall that the capacity of a set $E \subset \Omega $ is defined by
$$
\ca(E,\Omega) = \inf_{\mathcal U_E}\left\{\int\limits_{\Omega} |\nabla u|^2\,dxdy\right\},
$$
where $u\in\mathcal U_E\subset H^1_0(\Omega)$ iff $u\geqslant 1$ almost everywhere on an open neighbourhood of $E$. The standard mollification argument shows that it is sufficient to consider only test-functions from $C^\infty_0(\Omega)$ such that $0\leqslant u\leqslant 1$ and $u\equiv 1$ in a neighbourhood of $E$. 

If a certain property holds everywhere on $\Omega$ except for a subset $Z\subset \Omega$ such that $\ca(Z,\Omega)=0$, then we say that it   holds {\em quasi-everywhere} on $\Omega$ (or q.e. on $\Omega$). A subset $A\subset \Omega$ is {\em quasi-open} if for any $\varepsilon>0$ there exists an open set $A_\varepsilon$ such that $\ca(A\triangle A_\varepsilon,\Omega) <\varepsilon$. A function $f\colon\Omega\to \mathbb{R}$ is called {\em quasi-continuous} if for any $\varepsilon$ there exists a set $E_\varepsilon$ such that $\ca(E_\varepsilon,\Omega)<\varepsilon$
and $f\colon\Omega\setminus E_\varepsilon\to\mathbb{R}$ is a continuous function.

The next proposition, as well as  its corollaries below,  is  well-known.   Its proof, which essentially follows \cite[Section 3.3.4]{HP}, is included to make the presentation self-contained.
\begin{proposition}
\label{prop:capconv}
Let $\{u_n\}\subset C^\infty_0(\Omega)$ be a Cauchy sequence with respect to $H^1_0(\Omega)$-norm. Then there exists a subsequence $\{u_{n_k}\}$ converging pointwise outside of a set of capacity $0$. Moreover, the convergence is uniform outside of a set of arbitrarily small capacity.
\end{proposition}

\begin{proof}
Choose a subsequence $\{u_{n_k}\}$ such that $||u_{n_k}-u_{n_{k+1}}||^2_{H^1_0(\Omega)}\leqslant 2^{-3k}$. To simplify the notations we continue to denote that subsequence by $\{u_k\}$. Set 
$$
E_k = \left\{x\in \Omega|\,\,|u_k(x)-u_{k+1}(x)|>2^{-k}\right\}.
$$
Then one has that $2^k|u_k-u_{k+1}|\in \mathcal U_{E_k}$. Therefore,
$$
\ca(E_k,\Omega)\leqslant 2^{2k}\int\limits_\Omega |\nabla|u_k-u_{k+1}||^2\,dxdy\leqslant 2^{-k}.
$$
Set $F_m = \bigcup\limits_{k\geqslant m} E_k$, then one has
$$
\ca(F_m,\Omega)\leqslant \sum_{k=m}^\infty\ca(E_k,\Omega)\leqslant \sum_{k=m}^\infty 2^{-k}=2^{1-m}.
$$

If $x\in\Omega\setminus F_m$, then for all $k\geqslant m$ one has $|u_k(x) - u_{k+1}(x)|\leqslant 2^{-k}$. Therefore, for any $k,n\geqslant m$ one has 
$$
|u_k(x) - u_n(x)|\leqslant \sum_{i=k}^{n-1}|u_i(x)-u_{i+1}(x)|\leqslant 2^{1-k},
$$
i.e. $\{u_k\}$ converge uniformly on $\Omega\setminus F_m$.

Set $G = \bigcap\limits_m F_m$. Then outside of $G$ the sequence $\{u_k\}$ converges pointwise. Moreover, for any $n$ one has
$$
\ca(G,\Omega)\leqslant \ca(F_n,\Omega)\leqslant 2^{1-n}.
$$
Since $n$ is arbitrary, one has $\ca(G,\Omega) = 0$.
\end{proof}
The last assertion of Proposition \ref{prop:capconv} immediately implies:
\begin{corollary}
\label{cor:quasicont}
Any function $u\in H^1_0(\Omega)$ has a quasi-continuous representative $\widetilde u$. Moreover, $\widetilde u$ is unique up to a set of zero  capacity.
\end{corollary}
\begin{remark}
\label{qc:remark}
In the following,  when we work with a function $u\in H^1_0(\Omega)$, we always assume that $u$ is a quasi-continuous representative. 
\end{remark}
Under this convention the capacity can be computed using the following formula,
$$
\ca(E,\Omega) = \inf_{\mathcal V_E}\left\{\int\limits_{\Omega} |\nabla u|^2\,dxdy\right\},
$$
where $\mathcal V_E = \{u\in H^1_0(\Omega)|\,\, \text{$u$ is quasicontinuous},\,\,u\geqslant 1\,\, \mbox{q.e. on}\,\,E\}$. Furthermore, the sets $\{u>c\}$, $\{u<c\}$ are quasi-open. 

If $A\subset \Omega$ is quasi-open, we define $H^1_0(\Omega)$ to be a set of $u\in H^1_0(\Omega)$ such that $u=0$ q.e. in $\Omega\setminus A$. This definition agrees with the classical definition of $H^1_0(A)$ if $A$ is open. The following proposition is a straightforward corollary of Proposition \ref{prop:capconv} and Corollary \ref{cor:quasicont}. 
\begin{proposition}
\label{uniform2}
Suppose that $u_n\to u$ in $H^1_0(\Omega)$. Then there exists a subsequence $\{u_{n_k}\}$ such that for any $\varepsilon>0$ there exists a set $E_\varepsilon\subset \Omega$ with $\ca(E_\varepsilon,\Omega)<\varepsilon$ such that $u_{n_k}\rightrightarrows u$ in $\Omega\backslash E_\varepsilon$.  
\end{proposition}

Our next goal is to define quasi-continuous representatives for $H^1$-functions. Let $F\Subset \Omega$ be a compact subset of $\Omega$ with the {\em extension property}, i.e. all functions in $H^1(F)$ can be extended to $H^1(\mathbb{R}^2)$ and, as a result to $H^1_0(\Omega)$ as well. For example, all Euclidean balls and, in general, all Lipschitz domains possess the extension property. Let $u\in H^1(F)$ and let $v\in H^1_0(\Omega)$ be its extension. Then $v$ has a quasi-continuous representative $\widetilde v$ and, as a result, $\widetilde v|_F$ is a quasi-continuous representative of $u$. In the following we assume that functions from $H^1(F)$ are quasi-continuous.

\begin{corollary}
\label{uniform3}
Suppose that $p\in \Omega$ and let $B_r(p)\Subset \Omega$ be a ball of radius $r$. Suppose that $u_n\to u$ in $H^1(B_{r}(p))$. Then there exists a subsequence $u_{n_k}$ such that for any $\varepsilon>0$ there exists a set $E_\varepsilon\subset B_{r}(p)$ with $\ca(E_\varepsilon, \Omega)<\varepsilon$ such that $u_{n_k}\rightrightarrows u$ in $B_{r}(p)\backslash E_\varepsilon$.
\end{corollary}
\begin{proof}
Apply Proposition~\ref{uniform2} to the extensions of $u_n$.
\end{proof}

The following proposition  is a simplified version of the isocapacitory inequality (see, for example, \cite{Kok2}).

\begin{proposition}
Let $\mu\ne 0$ be a Radon measure on $\Omega$, $\mu(\Omega)<\infty$. Then for any $F\subset\Omega$ one has,
\begin{equation}
\label{isocapacitory:ineq}
\mu(F)\leqslant \frac{1}{\lambda_1^D(\Omega,\mu)}\ca(F,\Omega).
\end{equation}
\end{proposition}
\begin{proof}
Without loss of generality $\ca(F,\Omega)<\infty$. For a fixed $\varepsilon>0$ let $u_\varepsilon \in H^1_0(\Omega)$ be such that $u_\varepsilon|_F\geqslant 1$, $u_\varepsilon\geqslant 0$ and 
$$
\int\limits_\Omega|\nabla u_\varepsilon|^2\,dxdy\leqslant \ca(F,\Omega)+\varepsilon.
$$
Since $\mu\ne 0$, we have $\lambda_1^D(\Omega,\mu)<\infty$. Using $u_\varepsilon$ as a test-function for the Rayleigh quotient, one obtains,
$$
\lambda_1^D(\Omega,\mu)\mu(F)\leqslant \lambda_1^D(\Omega,\mu)\int\limits_\Omega u_\varepsilon^2\,d\mu\leqslant \int\limits_\Omega|\nabla u_\varepsilon|^2\,dxdy\leqslant \ca(F,\Omega)+\varepsilon.
$$
Since $\varepsilon$ is arbitrary, the proof is complete.
\end{proof}

%
%

\begin{lemma}
\label{C0H1:lemma}
Let $[a,b]\subset \mathbb{R}$ be a bounded interval. Then any $f\in H^1[a,b]$ is absolutely continuous and one has
$$
||f||^2_{C^0}\leqslant \max(b-a,2)||f||_{L^2} ||f||_{H^1}.
$$
\end{lemma}
\begin{proof}
Absolute continuity follows from the Rellich compactness theorem.

For two point $x,y\in [a,b]$ we write 
$$
f^2(x) - f^2(y) = \frac{1}{2}\int_x^y f(t)f'(t)dt\leqslant \frac{1}{2} ||f||_{L^2}||f'||_{L^2}.
$$
Integrating in $y$ we obtain for all $x$
$$
f^2(x)\leqslant \frac{b-a}{2}||f||_{L^2}||f'||_{L^2} + ||f||^2_{L^2}\leqslant \max(b-a,2)||f||_{L^2}||f||_{H^1},
$$
where in the last step we used inequality $\sqrt{c}+\sqrt{d}\leqslant 2\sqrt{(c+d)}$ for $c,d\geqslant 0$.
\end{proof}


The following proposition is essentially well-known and is a version of the ``absolute continuity on lines" property of functions  in  $H^1$ (see, for instance, \cite[Chapter 6]{HKST}).
However, our formulation differs from the standard one due to the fact that we always take a {\em particular representative} of an $H^1$ function, see Remark~\ref{qc:remark}.

\begin{proposition}
\label{H1Rademacher}
Let $R=[a,b]\times[c,d]\Subset\Omega$ be a rectangle. For any quasicontinuous $u\in H^1(R)$ let $X\subset [a,b]$ be a set consisting of points $x$ such that $u^x(y):= u(x,y)$ is absolutely continuous as a function of $y$ and $(u^x)'(y)\in L^2[c,d]$. Then $X$ is a set of full measure.
\end{proposition}
\begin{proof}
Let $v\in H^1_0(\Omega)$ be an extension of $u$. Let $\{v_n\}$ be a sequence of functions in $C^\infty_0(\Omega)$ that converge to $v$ in $H^1_0(\Omega)$. By Proposition~\ref{uniform2} we can assume that $v_n$ converge to $v$ outside of a set of capacity $0$. Let $u_n = v_n|_R$, then $u_n\in C^\infty(R)$ and $u_n$ converge to $u$ pointwise in $R$ outside of a set of capacity $0$.

By Fubini's theorem one has that the functions 
$$
\int_c^d (u_n(x,y)-u(x,y))^2 + \left(\frac{d}{dy}(u_n(x,y)-u(x,y))\right)^2\,dy
$$
tend to $0$ in $L^1[a,b]$ and, therefore, up to a choice of a subsequence, pointwise to $0$ for almost all $x\in[a,b]$. As a result, for almost all $x\in [a,b]$ the sequence $\{u_n^x\}$ converges in $H^1[c,d]$ and, therefore, uniformly as well. The problem is that, they may converge to a different representative of an $H^1$ function.

Let $X'$ be a set of $x\in [a,b]$ for which there exists $y$ such that $u_n(x,y)$ does not converge to $u(x,y)$. We claim that $X'$ has measure $0$. First, let us see why it completes the proof of the proposition. Indeed, for almost all $x$ $\{u^x_n\}$ converge and outside of $X'$ it converges to $u^x$ both in $C[c,d]$ and in $H^1[c,d]$. Therefore, for such $x$ the function $u^x$ is absolutely continuous and $(u^x)'\in L^2[c,d]$.

To show that $X'$ has measure $0$ we take $\phi_\varepsilon\in C^\infty_0(\Omega)$ such that $\phi_\varepsilon \geqslant 1$ on the set where $u_n(x,y)$ does not converge to $u(x,y)$ and $|\phi_\varepsilon|_{H^1(R)}<\varepsilon$. We let $c(x) = |\phi_\varepsilon^x|_{C[c,d]}$, $h(x) = |\phi_\varepsilon^x|_{H^1[c,d]}$. Then by Lemma~\ref{C0H1:lemma} one has
\begin{equation}
\label{eq'1}
c(x)^2\leqslant Ch(x)^2.
\end{equation}
Note that $c(x)\geqslant 1$ on $X'$. Therefore, integrating from $c$ to $d$ we obtain the following,
$$
|X'|\leqslant C|\phi_\varepsilon|^2_{H^1(R)}<C\varepsilon^2.
$$
Since $\varepsilon$ is arbitrary, we conclude that $X'$ has measure $0$.
\end{proof}
\subsection{Three auxiliary lemmas.} Recall that if $V\in L^\infty(\Omega)$, the first Dirichlet eigenvalue of $\Delta -V$ on $A$ is defined as
$$
\inf_{u\in H^1_0(A)} \frac{\int\limits_A |\nabla u|^2 - Vu^2\,dxdy}{\int_A u^2 \, dx dy}.
$$
\begin{lemma}
\label{lemma_H0}
Let $A\subset\Omega$ be quasi-open and let $V\in L^\infty(\Omega)$. Suppose that $\phi\in H^1(\Omega)$ is a weak solution to $\Delta\phi - V\phi = 0$ on $A$. Assume that the first eigenvalue of $\Delta - V$ in $H^1_0(A)$ is positive. Then for any $\psi\in H^1(\Omega)$ such that $(\phi - \psi)\in H^1_0(A)$ one has
$$
\int\limits_\Omega|\nabla\psi|^2 - V\psi^2\,dxdy\geqslant \int\limits_\Omega|\nabla\phi|^2 - V\phi^2\,dxdy
$$ 
\end{lemma}
\begin{proof}

From the eigenvalue condition we have that 
\begin{equation}
\label{H0:eq1}
\int\limits_\Omega |\nabla(\phi-\psi)|_g^2 - V(\phi-\psi)^2\,dxdy\geqslant 0.
\end{equation}
Moreover, pairing up the equation for $\phi$ with $\psi-\phi$, we obtain
$$
\int\limits_\Omega \nabla\phi\cdot\nabla(\psi-\phi) - V\phi(\psi-\phi)\,dxdy = 0,
$$
or, equivalently,
\begin{equation}
\label{H0:eq3}
\int\limits_\Omega \nabla\phi\cdot\nabla\psi - V\phi\psi\,dxdy = \int\limits_\Omega|\nabla\phi|^2 -V\phi^2\,dxdy,
\end{equation}

Summing up~\eqref{H0:eq1} and two copies of~\eqref{H0:eq3} yields
$$
\int\limits_\Omega |\nabla\phi|^2 + |\nabla\psi|^2 - V(\phi^2+\psi^2)\,dxdy\geqslant 2\int\limits_\Omega(|\nabla\phi|^2 - V\phi^2)\,dxdy.
$$
Rearranging the terms completes the proof.
\end{proof}


In what follows we use the following notations: $B_r(p)$ denotes the ball of radius $r$ with center at $p$; $S_r(p) = \partial B_r(p)$ is a circle of radius $r$ with center at $p$ and $A_{r,R}(p) = B_R(p)\setminus B_r(p)$ is an annulus around $p$.

\begin{lemma}
\label{reduction_to_H0}
Let $B_r(p)\subset B_R(p)\Subset\Omega$ be two balls. Let $u,v\in H^1(B_R(p))$ be such that $u|_{S_R(p)}>v|_{S_R(p)}$ and $u|_{S_r(p)}<v|_{S_r(p)}$. Then there exists a quasi-open set $A$ satisfying $B_r(p)\subset A\subset B_R(p)$ and $(u-v)|_A\in H^1_0(A)$. Moreover, the extension of $(v-u)|_A$ by zero to a function in $H^1_0(\Omega)$ is quasi-continuous.
\end{lemma}
\begin{proof}
Let $\hat u, \hat v\in H^1_0(\Omega)$ be extensions of $u,v$ respectively. Set $\hat A = \{v- u>0\}$ be a quasi-open subset of $\Omega$. Moreover, $\hat w=(v-u)_+$ is quasi-continuous such that $\hat w = 0$ on $S_R(p)$ and $\hat w=v-u$ on $S_r(p)$. Therefore, the function
$$
w(x) = 
\begin{cases}
0 &x\in\Omega\setminus B_R(p)\\
\hat w(x) & x\in B_R(p)\setminus B_r(p)\\
v(x)-u(x) & x\in B_r(p)
\end{cases}
$$
is quasi-continuous.
Then $A=\{w>0\}\cup B_r(p) = (\hat A\cap B_R(p))\cup B_r(p)$ is quasi-open and $w$ is an extension of $(v-u)|_A$ by zero.
\end{proof}

\begin{lemma}
\label{uniform1}
Let $U = A_{R,r}(p)\Subset\Omega$ be an annulus around a point $p\in \Omega$. Let $\{u_n\}\subset H^1(U)$ and $u\in H^1(U)$ be such that $\{u_n\}$ is equibounded in $H^1(U)$  and $u_n\to u$ in $L^2(U)$. Then up to a choice of a subsequence the set $\{\rho|\, u_n|_{S_\rho}\rightrightarrows u|_{S_\rho}\}$ has full measure in $(r,R)$.
\end{lemma}

\begin{proof} 
Without loss of generality, assume $u = 0$. 
Applying Proposition~\ref{H1Rademacher} in polar coordinates we see that the set $A_i$ of $\rho$ such that $u_i|_{S_\rho}$ is absolutely continuous on $S_\rho$ and its derivative is in $L^2(S_\rho)$ has full measure. Then $A = \cap_i A_i$ has full measure. The remainder of the proof is very similar to the proof of Proposition~\ref{H1Rademacher}.

For $\rho\in (r,R)$ define $h_i(\rho) = ||u_i||_{H^1(S_\rho)}$, $l_i(\rho) = ||u_i||_{L^2(S_\rho)}$ and $c_i(\rho) = ||u_i||_{C^0(S_\rho)}$. Application of Lemma~\ref{C0H1:lemma} yields that for all $\rho\in A$ one has
$$
c_i(\rho)^2\leqslant Cl_i(\rho)h_i(\rho). 
$$

By Fubini's theorem $l_i\in L^2(r,R)$ and $h_i\in L^2(r,R)$. Since $||u_i||_{H^1(A_{r,R})}\leqslant C_0$, one has $||h_i||_{L^2(r,R)}\leqslant C_0$. 

Integrating over $(r,R)$ yields
$$
||c_i||^2_{L^2(r,R)}\leqslant CC_0||l_i||_{L^2(r,R)}\to 0.
$$
Thus, $c_i^2\to 0$ in $L^2(r,R)$. Therefore, there exists a subsequence $i_k$ such that $c_{i_k}(\rho)\to 0$ for almost all $\rho\in (r,R)$. 
\end{proof}

\begin{corollary}
\label{corollary1}
Let $B_{R_0}(p)\Subset \Omega$. Suppose that $\{u_n\}$ is a bounded sequence in $H^1(B_{R_0}(p))$ such that $u_n\to u$ in $L^2(B_{R_0}(p))$, $u\in H^1(B_{R_0}(p))$. Then for any $r_0<R_0$ up to a choice of a subsequence there exist $r_0<r<R<R_0$ and a sequence $\{v_n\}\subset H^1(B_R)$ such that
\begin{itemize}
\item [1)] $||v_n - u||_\infty\to 0$
\item [2)] $v_n>u_n$ on $S_r(p)$
\item [3)] $v_n<u_n$ on $S_R(p)$
\item [4)] $||v_n-u||_{H^1(B_R)}\to 0$
\item [5)] $||\nabla v_n||_{L^2(B_R)}^2 - ||\nabla u||^2_{L^2(B_R)}\to 0$
\end{itemize} 
\end{corollary}

\begin{proof}
We apply Lemma~\ref{uniform1} to $A_{R_0,r_0}$. We get the subsequence such that the set $\{\rho|\, u_n|_{S_\rho}
\rightrightarrows u|_{S_\rho}\}$ is dense in $(r_0,R_0)$. Choose any $r<R$ from this set. Define a sequence $N(k)$ such 
that for $n\geqslant N(k)$ one has $|u_n- u|_{S_r\cup S_R}<2^{-k}$. For $n\in[N(k),N(k+1))$ we set $v_n$ to be $u + 
h_{k}$ where $h_{k}$ is a radial function with $h_{k}\equiv 2^{-k}$ on $B_r$ and linear on $A_{R,r}$ 
with $h_{k}(r) = -h_{k}(R) = 2^{-k}$. The properties 1)-3) follow immediately. Moreover,
\begin{equation*}
\int\limits_{B_R}|\nabla h_{k}|^2 = 2\pi\int_r^R \left(\frac{2^{1-k}}{R-r}\right)^2\rho\,d\rho = \pi 2^{2-2k} \frac{R+r}{R-r},
\end{equation*}  
and 4) follows.

In the following computation we use $||\cdot||$ to denote $||\cdot||_{{L^2(B_R)}}$. One has,
$$
\left|||\nabla v_n||^2 - ||\nabla u||^2\right| \leqslant ||\nabla v_n - \nabla u||(||\nabla u|| + ||\nabla v_n||)
$$
and 5) follows from 4) and equiboundedness of $\{v_n\}$ in $H^1(B_R)$.

\end{proof}
\subsection{Some properties of Radon measures.} 
Given a Radon measure $\nu$ on $\Omega$, one can define the corresponding Laplace eigenvalues with Dirichlet boundary conditions variationally by the following formula (see \cite{Kok2}):
\begin{equation}
\label{varD:Radon}
\lambda^D_k(M,\nu) = \inf_{E_k}\sup_{u\in E_k}\frac{\displaystyle\int_\Omega|\nabla u|^2\,dxdy}{\displaystyle\int_\Omega \,u^2\,d\nu},
\end{equation}
where the supremum is taken over $E_k \subset C_0^\infty(\Omega)$ which form  $k$-dimensional subspaces  in $L^2(d\nu)$. Furthermore, we assume the convention that the infimum over an empty set is equal to $+\infty$ and, as a result, all eigenvalues corresponding to the zero measure are equal to $+\infty$.

The following two auxiliary results on Radon measures will be used in Section \ref{regularity:sec}.
\begin{proposition}
\label{isocapacitory:corollary}
Suppose that $\mu$ is a finite Radon measure on $U\subset \Omega$ such that $\lambda_1^D(U, \mu)\geqslant C>0$. Then for any quasi-continuous $v\in H^1_0(U)$ such that $|v|\leqslant 1$ one has 
$$
\int_U|v|\,d\mu\leqslant \left(\mu(U) + \frac{1}{C}\right)\left(\int_U|\nabla v|^2\,dv_g\right)^{1/3}.
$$
\end{proposition}
\begin{proof}
Set $\gamma^3 = \int_U|\nabla v|^2\,dv_g$. We decompose $U= U_1\cup U_2$, where $U_1 = \{0\leqslant |v|< \gamma\}$ and $U_2 = \{\gamma\leqslant |v|\}$. Using $|v|/\gamma$ as a test-function we obtain
$$
\ca(U_2,U)\leqslant \gamma.
$$
Then the isocapacitory inequality~\eqref{isocapacitory:ineq} implies
$$
\mu(U_2)\leqslant \frac{\gamma}{C}.
$$
Finally, one obtains
$$
\int_U|v|\,d\mu\leqslant \int_{U_1} |v|\,d\mu + \int_{U_2} |v|\,d\mu\leqslant \gamma\mu(U_1) + \gamma/C
$$
\end{proof}

Finally, we recall the following proposition related to $*$-weak convergence of Radon measures.

\begin{proposition}
\label{weakproposition}
Let $V\Subset U$. Assume that $u_n$ converge weakly in $L^2(U)$ to $u$ and $u_n^2\,dv_g$ converge *-weakly to $d\nu$ as measures on $\bar V$. Then for any $W\subset V$ one has
$$
\nu(\bar W)\geqslant \int_Wu^2\,dv_g. 
$$
Moreover, equality holds for all $W\subset V$ iff $u_n$ converge to $u$ in $L^2(V)$.
\end{proposition}
\begin{proof}
Let us first recall that *-convergence implies that
\begin{equation}
\label{*weak}
\nu(\bar W)\geqslant\limsup\int_Wu_n^2\,dv_g\geqslant\liminf\int_Wu_n^2\,dv_g\geqslant \nu(W)
\end{equation}
for all $W\subset V$.

At the same time, for $u_n$ converge to $u$ weakly in $L^2(W)$. Therefore, by lower semicontinuity of the norm with respect to weak topology, one has
$$
\liminf\int_Wu_n^2\,dv_g\geqslant \int_W u^2\,dv_g.
$$
Combining it with the inequality~\eqref{*weak} we have a chain of inequalities
$$
\nu(\bar W)\geqslant \limsup\int_Wu_n^2\,dv_g\geqslant\limsup\int_W u_n^2\, dv_g\geqslant \int_Wu^2\,dv_g,
$$
which yields the first assertion.

Assume that we have an equality for all $W$. Therefore, in the previous chain all inequalities are actually equalities. In particular, one has that 
$$
\lim\int_Vu_n^2\,dv_g = \int_Vu^2\,dv_g,
$$
which yields $u_n\to u$ in $L^2(V)$.

Assume that $u_n\to u$ in $L^2(V)$, then it is easy to see that $d\nu = u^2\,dv_g$. Indeed, for any continuous function $\phi$ on $\bar V$ one has
$$
\left|\int_V \phi(u_n^2 - u^2)\,dv_g\right|\leqslant ||\phi||_\infty\left|\int_V u_n^2\,dv_g - \int_V u^2\,dv_g\right|\to 0.
$$
\end{proof}


\subsection{Note on cut-off functions}
\label{cutoff:section}
Given $R>r$, a straightforward computation shows that 
$$
\ca(B_r(p),B_R(p))\leqslant \frac{2\pi}{\ln\frac{R}{r}}.
$$
Indeed, the following function provides a suitable test-function
 \begin{equation}
\label{cutoffmodel}
f_{r,R}(\rho) = 
\begin{cases}
\displaystyle\frac{\ln\frac{\rho}{R}}{\ln\frac{r}{R}}&\text{on $A_{r,R}(p)$}\\
1&\text{on $B_r(p)$},\\
\end{cases}
\end{equation}
where $\rho$ is the radial coordinate.

In the following,  when we consider a cut-off function, we always mean the function defined by  \eqref{cutoffmodel} for a suitable choice of radii $r,R$.


\section{Regularity properties of the limiting measure}
\label{regularity:sec}
The goal of this section is to prove the following theorem.
\begin{theorem}
\label{regularity:thm}
Let $d\mu_k$ be the limiting measure as in property (6) of Theorem~\ref{thm:summary}. 
%
Then there exist at most $k$ points $p_1,\ldots, p_l$, $l\leqslant k$ and a harmonic map $\phi_k\colon M\to \mathbb{S}^{d-1}$ such that 
$$
d\mu_k = \frac{|\nabla\phi_k|^2}{\Lambda}dv_g + \sum_{i=1}^lw_i\delta_{p_i},
$$
where $w_i \ge 0$.
\end{theorem}
\begin{remark} We recall that  $|\nabla \phi_k|^2$  is understood in the sense of  Remark \ref{rem:vectornot}.  The zeros of the gradient of a harmonic map are isolated, and, therefore, the points where  $|\nabla \phi_k|=0$ correspond to the conical singularities of the limiting metric
 (see \cite[Section 5.3]{Kok2}).
\end{remark}

\subsection{Measure properties of $\mu_k$}
We first define the eigenvalues of a Radon measure on a surface similarly to~\eqref{varD:Radon}. Let $M$ be a surface and let $\mathcal C$ be a fixed conformal class on $M$ with a smooth background metric $g\in \mathcal C$.  Given a Radon measure $\nu$ on $M$, one can define the corresponding Laplace eigenvalues variationally by the following formula (see \cite{Kok2}):
\begin{equation}
\label{var:Radon}
\lambda_k(M,\nu) = \inf_{E_k}\sup_{u\in E_k}\frac{\displaystyle\int_M|\nabla u|^2\,dv_g}{\displaystyle\int_M \,u^2\,d\nu},
\end{equation}
where the supremum is taken over $E_k \subset C^\infty(M)$ which form  $(k+1)$-dimensional subspaces  in $L^2(d\nu)$.
Similarly, given a domain $U\subset M$, we define the Dirichlet eigenvalues $\lambda_k^D(U,\nu)$ by replacing $M$ by $U$ in the Rayleigh quotient and $C^\infty(M)$ by $C_0^\infty(U)$
 in the definition of $E_k$ and requiring $E_k$ to be $k$-dimensional instead $(k+1)$-dimensional. As before, all eigenvalues of the zero measure are assumed to be equal to $+\infty$.

\begin{lemma}
\label{disjoint}
Let $\nu$ be a Radon measure. Let $U_1,\ldots,U_{k+1}\subset M$ be a disjoint collection of open sets. Then for at least one $i$ one has 
$\lambda_1^D(U_i,\nu)\geqslant\lambda_k(M,\nu)$. 
\end{lemma}
\begin{proof}
If $\nu|_{U_i} = 0$ for some $i$, then it follows from \eqref{var:Radon} that $\lambda_1^D(U_i,\nu)=+\infty$, and the statement of the lemma is trivial.
Assume that $\nu|_{U_i}\ne 0$ for all $i=1,\dots, k+1$.
This condition implies that  $\lambda_1^D(U_i,\nu)<+\infty$. Arguing by contradiction, assume that for all $i$ one has $\lambda_1^D(U_i,\nu)<\lambda_k(M,\nu)$. 
Take the first Dirichlet eigenfunctions $f_i$ of $U_i$, continue them to the whole $M$ by zero,  and  apply the  variatonal principle for $\lambda_k(M,\nu)$ to the subspace of test-functions 
$\mathrm{span}\{f_1,\ldots,f_{k+1}\}$. We get a contradiction.
\end{proof}


Let $\nu_V$ be a Radon measure such that $d\nu_V=Vdv_g$ for some $V\in L^\infty(M)$. As before, we will write $\lambda_k(V):=\lambda_k(M,V)=\lambda_k(M,\nu_V)$
and $\lambda_k^D(\Omega,V):=\lambda_k^D(\Omega,\nu)$.
\begin{definition}
We say that a domain $\Omega\subset M$ satisfies $\sigma_k$-property for some $N\in \mathbb{N}$,  if  
$$
\lambda_1^D(\Omega, V_N)\geqslant\Lambda_k^N
$$ 
(cf. \cite[p. 19, condition ${\bf A}_{r,\epsilon}$]{Pet2}), 
or, equivalently, if the first eigenvalue of the Schr\"odinger operator $\Delta - \Lambda_k^NV_N$ on $H^1_0(\Omega)$ is non-negative. 
\end{definition}
In particular, if $V_N|_{\Omega} = 0$, then $\Omega$ satisfies $\sigma_k$-property for any $N$ since in this case $\lambda_1^D(\Omega, V_N)=+\infty$.

\begin{definition}
\label{def:good}
We say that the point $p$ is {\em good} if there exists an open neighbourhood $\Omega_p$ that satisfies $\sigma_k$-property for a subsequence $N_m\to\infty$. Otherwise, we say that
the point $p$ is {\em bad}.
\end{definition}
Note that   for any subdomain of $U_p\subset \Omega_p$ the $\sigma_k$-property is satisfied for the same subsequence $N_m$. Indeed, this immediately follows from the domain monotonicity for Dirichlet eigenvalues.

Let $G$ denote the set of all good points. Clearly, $G$ is an open set, since  if $p\in G$,  then $\Omega_p\subset G$.


\begin{proposition} 
\label{bad}
There exist $k$ points $p_1,\ldots,p_k$ such that $G\supset M\backslash\{p_1,\ldots,p_k\}$, i.e. all but at most $k$ points are good.
\end{proposition}
\begin{proof}
Assume the contrary, i.e. there exists $k+1$ bad points $p_1,\ldots, p_{k+1}$. Pick disjoint open neighbourhoods $U_i\ni p_i$. By Lemma~\ref{disjoint} applied to the measure $\nu_{V_N}$, for any $N \in \mathbb{N}$ there exists $i(N)$ such that $\lambda_k(U_i,V_N)\geqslant\Lambda_k^N$. Therefore, there exists a subsequence $N_m$ such that $i(N_m)\equiv i_0$ is constant, i.e. the point $p_{i_0}$ is good. We arrive at a contradiction.
\end{proof}

For the remainder of this section for each point $p$ we fix a small open neighbourhood $\Omega_p$ such that $g$ is conformally flat on $\Omega_p$. This way we can use the capacity estimates of Section~\ref{preliminaries:sec} with $\Omega = \Omega_p$ whenever we are working in the neighbourhood of $p$.

\subsection{Strong convergence in $H^1$ in a neighbourhood of a good point}
%

\begin{proposition}
\label{regularity}
Given a good point $p$,  there exists a neighborhood $X \ni p$ and a sequence $N_m\to \infty$ such that $\phi_{N_m,k}\to \phi_k$ in $H^1(X)$.
\end{proposition}
\begin{remark}
\label{rem:epsreg}
This proposition could be viewed as an $\varepsilon$-regularity-type  theorem (see \cite{ColdMin}) in the following sense. We claim that if  $\dfrac{1}{\lambda_1^D(\Omega)}$ is small, then
weak $H^1$ convergence of $\phi_{N_m,k}$  implies strong convergence in $H^1(X)$ for some $X \subset \Omega$. 
\end{remark}
\begin{proof}


Assume, by contradiction, that there is no such neighborhood.
Let $U$ be a neighbourhood of $p$ such that $U$ satisfies $\sigma_k$-property for all large enough $N$. Let $p\in Y\Subset U$. Choose a subsequence $N_m \to \infty$ such that $|\nabla \phi_{N_m,k}|^2\,dv_g$ converges *-weakly to $d\nu$ as measures on $\bar Y$.  Indeed, such a subsequence exists because $\phi_{N_m,k}$ is weakly convergent in $H^1$ by  assertion (4) in Theorem \ref{thm:summary}, and hence the measures  $|\nabla \phi_{N_m,k}|^2\,dv_g$  are bounded and contain a *-weakly convergent subsequence.
For any $X\subset Y$ the restrictions of these measures  
converge as measures on $\bar X$. We claim that there exists a point $q\in\supp(d\nu - |\nabla\phi_k|^2\,dv_g)\cap Y$. Indeed, by Proposition~\ref{weakproposition} applied to $|\nabla \phi_{N_m,k}|$ and $V$ one has $\supp(d\nu - |\nabla\phi_k|^2\,dv_g)\ne \emptyset$. If the support is concentrated on $\partial Y$ then for any $p\in X\Subset Y$ one has $d\nu|_{X} = |\nabla\phi_k|^2\,dv_g|_X$. Then by Proposition~\ref{weakproposition} $|\nabla\phi_{N_m,k}|\to |\nabla\phi_k|$ in $L^2(X)$ which implies that for any $i$ one has $|\nabla u_{N_m,k}^i|^2\to |\nabla u_k^i|^2$ in $L^1(W)$. Since $\phi^N_k\to\phi_k$ in $L^2(M)$, this implies that components converge in $H^1(X)$ in contradiction with our initial assumption. 
%
Let 
\begin{equation}
\label{suppq}
q\in \supp(d\nu - |\nabla \phi_k|^2\,dv_g)\cap Y,
\end{equation}
 and let $r<R$ be such that $B_R(q)\subset Y$.
In the argument below  we will be consequently refining the disks $B_r(q)\subset B_R(q)$ and the sequence $\{N_m\}$ by picking new $r_0,R_0$ satisfying $r<r_0<R_0<R$ in a way that they satisfy more and more conditions. Each condition will be preserved under such refinement.
After each refinement we will omit the index $0$ from our notations and keep the notation $N_m$ for the subsequence. 
This way eventually  we obtain the disks $B_r(q)\subset B_R(q)$ and a subsequence $N_m\to \infty$ satisfying conditions (C1) -- (C4) below.

In view of \eqref{suppq}, there exists $\epsilon>0$ such that
\begin{equation}
\label{condition1}
\nu(B_{r}(q)) - \int_{B_{r}}|\nabla \phi_k|^2\,dv_g>2\varepsilon.
\tag{C1}
\end{equation}

The condition~\eqref{condition2} is as follows,
\begin{equation}
\label{condition2}
\int_{A_{r,R}(q)}|\nabla \phi_{N_m,k}|^2\,dv_g <\frac{\varepsilon}{3}
\tag{C2}
\end{equation}
for all $N_m$. In order to satisfy this condition we divide the initial annulus $A_{r,R}$ into $K$ sub-annuli. Since the sequence $\phi_{N_m,k}$ is bounded in $H^1(M)$, one can take $K$ so large that for each $N_m$ at least one of subannuli satisfies~\eqref{condition2}. Since there are finitely many such subannuli one can choose a subsequence such that condition~\eqref{condition2} is satisfied for all members of the subsequence.

Properties~\eqref{condition1} and~\eqref{condition2} imply that for large enough $N_m$,
\begin{equation}
\label{condition3}
\int_{B_r}|\nabla \phi_{N_m,k}|^2\,dv_g\geqslant \int_{B_R}|\nabla \phi_k|^2\,dv_g + \varepsilon
\tag{C3}
\end{equation}
Indeed, condition~\eqref{condition1} and the last inequality in formula \eqref{*weak} yield
$$
\liminf\int_{B_r}|\nabla \phi_{N_m,k}|^2\,dv_g - \int_{B_r}|\nabla \phi_k|^2\,dv_g\geqslant \nu(B_r)- \int_{B_r}|\nabla \phi_k|^2\,dv_g>2\varepsilon.
$$
Therefore, for large enough $N_m$ one has
$$
\int_{B_r}|\nabla \phi_{N_m,k}|^2\,dv_g\geqslant \int_{B_r}|\nabla \phi_k|^2\,dv_g + 2\varepsilon.
$$
At the same time, property~\eqref{condition2} together with the fact that $\phi_{N_m,k}$ converge weakly in $H^1$ to $\phi_k$ 
(and thus $\lim\inf_{N_m\to \infty}\|\phi_{N_m,k}\|_{H^1} \ge \|\phi_k\|_{H^1}$) implies that
$$
\int_{A_{r,R}}|\nabla \phi_k|^2\,dv_g<\frac{\varepsilon}{3}.
$$
Summing this up with the previous inequality yields~\eqref{condition3}.

Finally, we would like to ensure that the annulus $A_{r,R}$ satisfies 
\begin{equation}
\label{condition4}
\int_{A_{r,R}}V_N\,dv_g\leqslant \frac{\varepsilon}{9\Lambda_k d}.
\tag{C4}
\end{equation}
It is achieved  in the same way as for condition~\eqref{condition2}.

At this point we apply Corollary~\ref{corollary1} to each component $u^i_{N_m,k}$ and balls $B_r\subset B_R$ to get a sequence $v^i_{N_m,k}$. We then apply Lemma~\ref{reduction_to_H0} to $v^i_{N_m,k}$ and $u^i_{N_m,k}$ to get a quasi-open set $A^i_{N_m}$. 

Since $B_R$ satisfies $\sigma_k$-property for all $N_m$, we can apply Lemma~\ref{lemma_H0} with $\phi = u^i_{N_m,k}$, $\psi =v^i_{N_m,k}$, $V = V_{N_m,k}$ 
and $A = A^i_{N_m}$ to conclude
\begin{multline*}
\int_{A^i_{N_m}}(|\nabla  v^i_{N_m,k}|^2 - \Lambda_k^{N_m}( v^i_{N_m,k})^2 V_{N_m,k})\,dv_g\geqslant \\ \int_{A^i_{N_m}}(|\nabla  u^i_{N_m,k}|^2 - \Lambda_k^{N_m} ( u^i_{N_m,k})^2V_{N_m,k})\,dv_g.
\end{multline*}
Rearranging yields
\begin{multline}
\label{eq1}
\int_{A^i_{N_m}} |\nabla v^i_{N_m,k}|^2\,dv_g + \Lambda_k^{N_m}\int_{A^i_{N_m}}\left(( u^i_{N_m,k})^2 - ( v^i_{N_m,k})^2\right)V_{N_m,k}\,dv_g \geqslant \\ \int_{A_{N_m}^i} |\nabla  u^i_{N_m,k}|^2\,dv_g.
\end{multline}
Analyzing the term in the middle, we note that the integrand is bounded in absolute value by $3$, and since $B_r\subset A_i^{N_m}\subset B_R$,  condition~\eqref{condition4} implies
$$
\left|\Lambda_k^{N_m}\int_{B_R\backslash\Omega^i_{N_m}}\left(( u^i_{N_m,k})^2 - ( v^i_{N_m,k})^2\right)V_{N_m,k}\,dv_g\right|\leqslant \frac{\varepsilon}{3d}.
$$
Therefore, in inequality~\eqref{eq1} one can replace the domain of integration in the l.h.s by $B_R$ and in the r.h.s by $B_r$ with a loss of at most $\frac{\varepsilon}{3d}$ to obtain
\begin{multline*}
\int_{B_R} |\nabla  v^i_{N_m,k}|^2\,dv_g + \Lambda_k^{N_m}\int_{B_R}\left(( u^i_{N_m,k})^2 - ( v^i_{N_m,k})^2\right)V_{N_m,k}\,dv_g + \frac{\varepsilon}{3d}\geqslant \\ \int_{B_r} |\nabla u^i_{N_m,k}|^2\,dv_g.
\end{multline*}

Recall that $ v^i_{N_m,k}$ were constructed using Corollary~\ref{corollary1}. By properties 1) and 5) for large enough $N_m$ one can replace $ v^i_{N_m,k}$ by $u^i_k$ in the left hand side of the previous inequality again with a loss of at most $\frac{\varepsilon}{3d}$ to obtain 
$$
\int_{B_R} |\nabla u^i_k|^2\,dv_g + \Lambda_k^{N_m}\int_{B_R}\left(( u^i_{N_m,k})^2 - (u^i_k)^2\right)V_{N_m,k}\,dv_g + \frac{2\varepsilon}{3d}\geqslant \int_{B_r} |\nabla  u^i_{N_m,k}|^2\,dv_g.
$$
Finally, we sum this inequality over all $i$ and use that $|\phi_{N_m,k}| \leqslant |\phi_k| = 1$ $dv_g$-a.e., which implies that the middle term on the left-hand side is nonpositive. 
Thereforem, we obtain
$$
\int_{B_R} |\nabla \phi_k|^2\,dv_g + \frac{2\varepsilon}{3}\geqslant \int_{B_r} |\nabla \phi_{N_m,k}|^2\,dv_g.
$$
Combining it with property~\eqref{condition3} we arrive at a contradiction.
\end{proof}
\begin{corollary}
\label{cor:uniform}
Let $p\in M$ be a good point. 
Then there exists a subsequence $N_m\to\infty$ such that for any $\varepsilon>0$ there exists a set $E_\varepsilon\subset B_r(p)\subset\Omega$ with $\ca(E_\varepsilon, \Omega)<\varepsilon$ such that $\phi_{N_m,k}\rightrightarrows \phi_k$ in $B_{r}(p)\backslash E_\varepsilon$.
\end{corollary}
\begin{proof}
The result follows immediately from Proposition \ref{regularity} and Corollary~\ref{uniform3}.
\end{proof}
Recall that  $G$ denotes the open set of all good points on $M$. 
\begin{proposition}
\label{weaksoleq}
Let $\psi\in C^\infty_0(G)$. Then for any $i=1,\ldots,d$,  one has
\begin{equation}
\label{limit}
\int_G \nabla\psi\cdot\nabla u^i_k\,dv_g = \Lambda_k\int_G \psi u^i_k\,d\mu_k.
\end{equation}
\end{proposition}
\begin{remark}
Proposition \ref{weaksoleq}  means that  the functions $u^i$ are weak solutions in $G$ of the equation
$$ 
\Delta_g u^i_k dv_g= \Lambda_k u^i_k d\mu_k.
$$
\end{remark}
\begin{proof}
Applying partition of unity, it is enough to prove the proposition for $\psi$ with support in a small neighbourhood of a good point. In particular, for any point $p\in G$,  let $E_\varepsilon\Subset B_r(p)$ be a pair of neigbourhoods such that Proposition~\ref{regularity} holds. We set $X = E_\varepsilon$ and $Y=B_r(p)$. Without loss of generality, $\sigma_k$-property holds on $Y$ for all large $N$. Fix $\varepsilon>0$ and let $\{N_m\}$ and $E_\varepsilon\subset W$ denote the corresponding subsequence and the subset. Assume $\supp\, \psi\subset X$.

Let $\beta_\varepsilon\in C^\infty_0(\Omega)$ be a smooth function such that $0\leqslant\beta_\varepsilon\leqslant 1$, $\beta_\varepsilon = 1$ on $X$ and
\begin{equation}
\label{betaeps}
\int_\Omega|\nabla \beta_\varepsilon|^2\,dv_g = \varepsilon;
\end{equation}
the latter is possible by Corollary \ref{cor:uniform}.
 Recall that $\Delta\phi_{N_m,k} = \Lambda^{N_m}_kV_{N_m,k}\phi_{N_m,k}$.  Pairing it with $\psi(1-\beta_\varepsilon)$ we obtain
\begin{equation*}
\int_X \nabla u^i_{N_m, k}\cdot\nabla(\psi(1-\beta_\varepsilon))\,dv_g = \Lambda^{N_m}_k\int_X \psi(1-\beta_\varepsilon) u^i_{N_m, k}  V_{N_m,k}\,dv_g.
\end{equation*}
We pass to the limit $N_m\to\infty$. The limit of the l.h.s is easy, since $\phi_{N_m,k}\rightharpoonup\phi_k$ in $H^1(M)$. For the r.h.s. we write
\begin{equation*}
\begin{split}
&\left|\Lambda_k^{N_m}\int_X \psi(1-\beta_\varepsilon)  u^i_{N_m, k} V_{N_m,k}\,dv_g - \Lambda_k\int_X\psi(1-\beta_\varepsilon) u^i_k\, d\mu_k\right|\leqslant \\ &
 \left|\int_X\psi(1-\beta_\varepsilon)(\Lambda_k^{N_m}   u^i_{N_m, k} - \Lambda_k u^i_k)V_{N_m,k}\,dv_g\right| + 
\left| \int_X\psi(1-\beta_\varepsilon)\Lambda_k u^i_k\,(V_{N_m,k}dv_g - d\mu_k)\right|.
\end{split}
\end{equation*}
The first summand tends to zero,  because on $\supp(\psi(1-\beta_\varepsilon))$ one has $\Lambda_k^{N_m} u^i_{N_m, k} \rightrightarrows \Lambda_k u^i_k$  by Proposition~\ref{regularity} . 
The second summand tends to zero by the definition of the *-weak convergence of measures (since $u^i_k$ are continuous on $\supp(\psi(1-\beta_\varepsilon))$). Thus, we obtain
\begin{equation}
\label{pre-limit}
\int_X\nabla u^i_k\cdot\nabla(\psi(1-\beta_\varepsilon))\,dv_g = \Lambda_k\int_X u^i_k\psi(1-\beta_\varepsilon)\,d\mu_k.
\end{equation}

We claim that passing to the limit as $\varepsilon\to 0$ in~\eqref{pre-limit} yields~\eqref{limit}. We prove this in two steps.
First, note that 
\begin{equation}
\label{claim11}
\int_X\nabla u^i_k\cdot\nabla(\beta_\varepsilon\psi)\,dv_g\to 0.
\end{equation}
Indeed, 
$$
\left|\int_X\nabla u^i_k\cdot\nabla(\beta_\varepsilon\psi)\,dv_g\right|\leqslant ||\nabla u^i_k||_{L^2(X)}(\sup|\psi|||\nabla \beta_\varepsilon||_{L^2(X)} + \sup|\nabla\psi|||\beta_\varepsilon||_{L^2(X)}),
$$
which tends to zero as $\varepsilon\to 0$ in view of \eqref{betaeps} and the Friedrichs inequality.
Second, we claim that
\begin{equation}
\label{claim21}
\int_X u^i_k \, \psi \, \beta_\varepsilon\,d\mu_k\to 0.
\end{equation} 
Since $|u^i_k\psi|$ is bounded, it is sufficient to show that
$$
\int_\Omega|\beta_\varepsilon|\,d\mu_k\to 0.
$$
Using the  $\sigma_k$-property and the upper semicontinuity of eigenvalues, one obtains $\lambda_1^D(Y,\mu_k)\geqslant~\Lambda_k$. Therefore, the limit \eqref{claim21} follows 
Proposition~\ref{isocapacitory:corollary}. 

Combining \eqref{claim11} and \eqref{claim21} we obtain \eqref{limit} from \eqref{pre-limit}, and this completes the proof of the proposition.

%
%
%
\end{proof}

%

Recall that $|\phi_k|^2\equiv 1$ $dv_g$-a.e. If we informally apply $\Delta$ to this equality, using Proposition \ref{weaksoleq}  we obtain
$$
|\nabla \phi_k|^2 dv_g= \Lambda_k\,d\mu_k
$$
weakly on $G$.  
The goal of the next proposition is to make this computation rigorous.
\begin{proposition}
\label{dmuk}
One has on $G$
$$
d\mu_k = \frac{|\nabla\phi_k|^2}{\Lambda_k}\,dv_g.
$$
\end{proposition}

\begin{proof} 
Let $V\Subset U\Subset G$ be such that $U$ satisfies $\sigma_k$-condition for all large enough $N$. It is sufficient to check that for any $\psi\in C^\infty_0(V)$ one has
$$
\Lambda_k\int_V\psi\,d\mu_k = \int_V\psi|\nabla\phi_k|^2\,dv_g.
$$

Let $\rho_{m,i}\in C_0^\infty(U)$ be such that their restrictions to $V$ converge in $H^1(V)$ to $u^i_k|V$. Moreover, the family $\{\rho_{m,i}\}$ can be chosen to be equibounded. Then, by Proposition~\ref{isocapacitory:corollary} $\rho_{m,i}|_V$ converge to $u^i_k$ in $L^2(V,d\mu_k)$. Moreover, since $|\psi|$ and $|\nabla\psi|$ are bounded one has $\rho_{m,i}\psi\to u^i_k\psi$ in $H^1_0(V)$ and in $L^2(V,d\mu_k)$. Therefore, applying~\eqref{limit} with test function $\rho_{m,i}\psi$ and passing to the limit $m\to\infty$ yields
$$
\int_V\nabla u^i_k\cdot\nabla(u^i_k\psi)\,dv_g = \Lambda_k\int_V (u^i_k)^2\psi\,d\mu_k.
$$ 
We then have
$$
\int_V\nabla u^i_k\cdot\nabla(u^i_k\psi)\,dv_g = \int_V|\nabla u^i_k|^2\psi\,dv_g + \frac{1}{2}\int_V\nabla \left( (u^i_k)^2\right)\cdot\nabla\psi\,dv_g.
$$
Summing up over $i$ we obtain
$$
\int_V|\nabla\phi_k|^2\psi\,dv_g + \frac{1}{2}\int_V\nabla(|\phi_k|^2)\cdot\nabla\psi\,dv_g = \Lambda_k\int_V|\phi_k|^2\psi\,d\mu_k
$$
As $|\phi_k|^2 = 1$ $dv_g$-a.e. for the second summand on the l.h.s we have 
$$
\int_V\nabla(|\phi_k|^2)\cdot\nabla\psi\,dv_g = \int_V\Delta\psi\,dv_g = 0,
$$
since $\psi\in C^\infty_0(V)$. Thus we arrive at
$$
\int_V|\nabla\phi_k|^2\psi\,dv_g = \Lambda_k\int_V|\phi_k|^2\psi\,d\mu_k.
$$
Finally, we note that by the isocapacitory inequality and $\sigma_k$ condition, we have that for any $F\subset V$ one has
$$
\mu_k(F)\leqslant \frac{1}{\Lambda_k}\ca(F,U).
$$
Since $\phi_k$ is quasicontinuous we have that $|\phi_k|^2 = 1$ q.e. and we conclude that $|\phi_k|^2 = 1$ $d\mu_k$-a.e. on $V$, which concludes the proof.
\end{proof}

\noindent{\em Proof of Theorem \ref{regularity:thm}.}
Substituting the expression for $d\mu_k$ obtained in Proposition \ref{dmuk} into formula \eqref{limit} we  show that on the set $G$  of good points, the map  $\phi_k$ is a weak solution of
$$
\Delta\phi_k = \frac{|\nabla \phi_k|^2}{\Lambda_k}\phi_k,
$$
i.e. a weakly harmonic map to $\mathbb{S}^{d-1}$. By a regularity theorem of H\'elein \cite{Helein} this implies that $\phi_k \in C^{\infty}(G,\mathbb{S}^{d-1})$. Since $G$ is equal to $M$ without a finite number of points, one can apply the removable singularity theorem for harmonic maps  \cite{SacksUhl} to obtain a harmonic map $\phi_k\colon M\to \mathbb{S}^{d-1}$
(we note that similar regularity arguments were used in \cite[Section 4.4]{Kok2} and \cite[Section 6.1]{Pet2}).
Therefore, 
$$
d\mu_k = \frac{|\nabla\phi_k|^2}{\Lambda_k}dv_g + \sum\limits_{i=1}^k w_i \delta_{p_i},
$$
where $p_i$ are the bad points from Proposition~\ref{bad}. This completes the proof of Theorem \ref{regularity:thm}.  \qed

\section{Atoms}
\label{sec:atoms}

In this section we focus on atoms arising at bad points $p_i$. We perform a procedure reminiscent of the bubble tree construction, see e.g.~\cite{Parker}. 

Fix a bad  point $p_i$ of weight $w_i$. Choose a small renormalization constant $C_R>0$ which will be specified later.
To simplify  notation, in the following we omit the subscript $i$. Recall that we have a sequence $N_{m}\to\infty$ and the corresponding maps $\phi_{N_m,k}$. Denote by $d\nu^r$ the regular part of the measure $d\mu_k$.

\subsection{Bubble tree construction} 
\label{subsec:bubbles}
Assume $w>C_R$. We work in a small neighbourhood of $p$, where the metric $g$ is conformally flat. In what  follows,  the distances in this neighbourhood are measured with respect to the flat metric $g_0$. 
%
%
%

Let $1\gg\varepsilon_m>0$ be a sequence of numbers, where $a_m\gg b_m$ means that $\frac{b_m}{a_m}\to 0$. Without loss of generality we may assume that the ball $B_m := B_{2\varepsilon_m}(p)$ can be identified with a subset of $\mathbb{R}^2$.
%
%
%
Choose $\delta'_m\ll \varepsilon_m$.
\begin{lemma}
\label{tree:lemma0}
Up to a choice of a subsequence one has $\mu_{N_m,k}(B_m) = w+o(1)$; $\mu_{N_m,k}(B_m\setminus B_{\delta'_m/4}(p)) = o(1)$.
\end{lemma}
\begin{proof}
Let $a_m =\frac{1}{w}\nu^r(B_m)\ll 1$. Then $\mu_k(B_m)=(\nu_r+w\delta_p)(B_m) = w(1+a_m)$. Thus, by Theorem \ref{thm:summary}(6)  and Theorem \ref{regularity:thm},
for a fixed $n$ there exists $m_n$ such that for all $m\geqslant m_n$ one has
\begin{equation*}
\begin{split}
w(1-a_n)\leqslant& \mu_{N_m,k}(B_{\delta'_n/4}(p))\leqslant w(1+2a_n);\\
&\mu_{N_m,k}\left(B_n\setminus B_{\delta'_n/4}(p)\right)\leqslant 2w a_n.
\end{split}
\end{equation*}
We define a subsequence $n_l = \max(m_l,n_{l-1}+1)$, which we rename $\{m\}$ to simplify notation. For this subsequence one has
\begin{equation*}
\begin{split}
w(1-a_m)\leqslant& \mu_{N_m,k}(B_{\delta'_m/4}(p))\leqslant w(1+2a_m);\\
&\mu_{N_m,k}\left(B_m\setminus B_{\delta'_m/4}(p)\right)\leqslant 2wa_m.
\end{split}
\end{equation*}
The second inequality yields $\mu_{N_m,k}\left(B_m\setminus B_{\delta'_m/4}(p)\right)=o(1)$ and summing up the two inequalities gives $\mu_{N_m,k}(B_m) = w+o(1)$.
This completes the proof of the lemma.
\end{proof}

Let $m$ be large enough so that $\mu_{N_m,k}(B_m\setminus B_{\delta'_m/4}(p))<\min(C_R,w-C_R)$. For each $x\in B_{\varepsilon_m}(p)$ let $\alpha(x)$ be such that
$$
\mu_{N_m,k}(B_m\setminus B_{\alpha(x)}(x))=C_R.
$$
Let $c_m\in B_{\varepsilon_m}(p)$ be any point such that 
$$
\alpha(c_m)<2\inf\limits_{x\in B_{\varepsilon_m}(p)}\alpha(x)
$$
and set $\alpha_m = \alpha(c_m)$.
\begin{lemma}
\label{tree:lemma1}
One has $|c_m|,\alpha_m\ll\varepsilon_m$, and hence $B_{\varepsilon_m}(c_m) \subset B_m$ (see Figure 1, left).  
In addition, $\mu_{N_m,k}(B_m\setminus B_{\delta'_m}(c_m)) = o(1)$.
\end{lemma}
\begin{proof}
By Lemma~\ref{tree:lemma0} one has $\alpha(p)\leqslant \frac{\delta'_m}{4}$. Therefore, $\alpha_m<\frac{\delta'_m}{2}$.

Given any $x \in B_m$ such that  $|x|>\frac{3\delta'_m}{4}$, we have  $B_{\delta'_m/2}(x)\subset B_m\setminus B_{\delta_m'/4}(p)$. As a result, for large $m$ one has
$$
\mu_{N_m,k}(B_m\setminus B_{\delta'_m/2}(x))>w-(w-C_R)=C_R,
$$
i.e. $\alpha(x)>\frac{\delta'_m}{2}$. The latter implies $x\ne c_m$ and hence  $|c_m|\leqslant \frac{3\delta'_m}{4}$. In particular, $B_{\delta'_m/4}(p)\subset B_{\delta'_m}(c_m)$ and the last assertion of the lemma follows from  Lemma~\ref{tree:lemma0}.
\end{proof}

Define $\delta_m = \sqrt{\delta'_m\varepsilon_m}$.
\begin{figure}%
    \centering
    \subfloat{{\includegraphics[align=c,width=5cm]{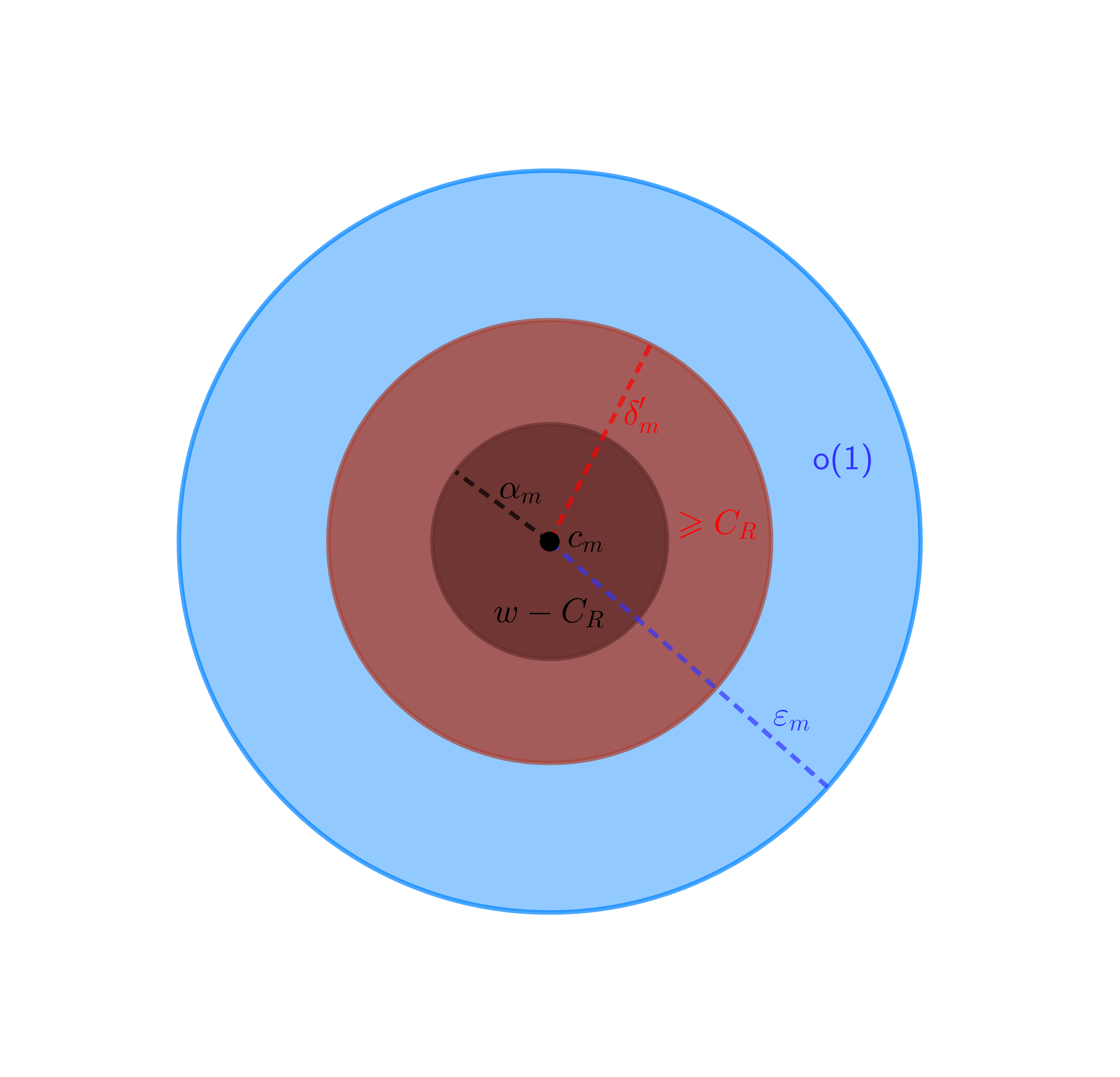} }}%
    \qquad
    \subfloat{{\includegraphics[align=c,width=6cm]{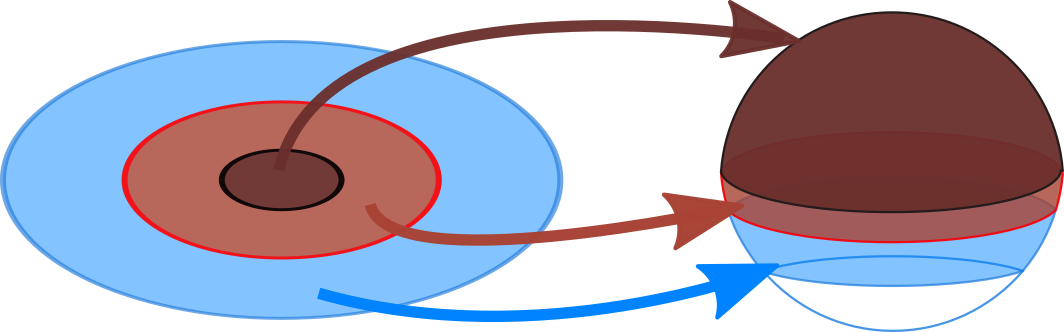} }}%
    \caption{ On the left,  the neighbourhood $B_{\varepsilon_m}(c_m)$ of a bad point with the center  $c_m$,  and  the masses of  the corresponding regions.  On the right,  the map $R_m$ from $B_m$ onto the sphere. 
The image of the dark brown disk is the northern hemisphere. As $m\to\infty$, the image of the middle ring tends to the southern hemisphere, and the image of the blue ring gets squeezed 
into the south pole.}
    \label{fig1}%
\end{figure}

We then have a map $R_m$ defined on $B_m$ as follows: 
$$
R_m(x) = \pi(\alpha_m^{-1}(x-c_m)),
$$ 
where $\pi$ is the inverse stereographic projection from the south pole to the equatorial plane, so  that  $R_m(B_{\alpha_m}(c_m))$ is the northern hemisphere (see Figure 1, right). Let $\Omega_m\subset\mathbb{S}^2$ be the image of $B_{\delta_m}(c_m)$ under $R_m$. Since $\varepsilon_m\gg\alpha_m$ one has that $\bigcup_m\Omega_m = \mathbb{S}^2\setminus\{S\}$, where $S$ is the south pole. 

We further push-forward the  measures $\mu_{N_m,k}$ by $(R_m)_*$ to measures $d\tilde\mu_{N_m,k} = \tilde V_{N_m,k}\,dv_{g_{\mathbb{S}^2}}$ and pull-back the maps $\phi_{N_m,k}$ to  maps $\tilde\phi_{N_m,k}$ on 
$\Omega_m$ satisfying
\begin{equation}
\label{deltatilde}
\Delta_{g_{\mathbb{S}^2}}\tilde\phi_{N_m,k} = \Lambda^{N_m}_k\tilde V_{N_m,k}\tilde\phi_{N_m,k}.
\end{equation}
Extend $\widetilde \mu_{N_m,k}$ by $0$ to the whole $\mathbb{S}^2$.
Let $\widetilde \mu$ be a *-weak limit of $\widetilde \mu_{N_m,k}$. 

\begin{lemma}
The measure $\widetilde \mu$ satisfies $\widetilde \mu(\mathbb{S}^2) = w$, and $\widetilde\mu$-measure of the southern hemisphere is at least $C_R$.
\end{lemma}
\begin{proof}
The statement follows from Lemma~\ref{tree:lemma1} after noting that since $\mu_{N_m,k}(B_m\setminus B_{\delta'_m}(c_m)) = o(1)$, the measures $(R_m)_*(\mu_{N_m,k}|_{B_m})$ and $\widetilde\mu_{N_m,k}$ have the same $*$-weak limit.
\end{proof}

Define $\tau_S = \widetilde\mu (S)$.
\begin{lemma}
\label{neck:lemma}
Assume $\tau_S\ne 0$. Then up to a choice of a subsequence there exists $\gamma'_m,\beta_m$ such that
\begin{itemize}
\item $\alpha_m\ll\beta_m\ll\gamma'_m\ll\delta'_m$;
\item $\mu_{N_m,k}(B_{\delta_m}(c_m)\setminus B_{\gamma'_m}(c_m)) = \tau_S+o(1)$;
\item $\mu_{N_m,k}(B_{\gamma'_m}(c_m)\setminus B_{\beta_m}(c_m)) = o(1)$.
\end{itemize}
\end{lemma}
\begin{proof}
Let $\widetilde B_r(S)$ be a neighbourhood of the south pole $S\in \mathbb{S}^2$ defined as the complement of $B_{r^{-1}}(N)$, where the distance is measured in the metric $(\pi^{-1})^*g_{\mathbb{R}^2}$. 

Let $1\gg\widetilde\beta_m\gg \widetilde \gamma'_m$ and let $a_m = \widetilde\mu(\widetilde B_{\widetilde\beta_m}(S))-\tau_S\ll 1$. Then for a fixed $n$ there exists $m_n$ such that for all $m\geqslant m_n$ one has
$$
n\frac{\alpha_m}{\delta'_m}\leqslant \widetilde\gamma'_n
$$
$$
\tau_S-a_m\leqslant \widetilde\mu_{N_m,k}(\widetilde B_{\widetilde \gamma'_n}(S))\leqslant \tau_S+a_m
$$
$$
\widetilde\mu_{N_m,k}(\widetilde B_{\widetilde\beta_n}(S)\setminus \widetilde B_{\widetilde \gamma'_n}(S))\leqslant 3a_m
$$
Define a subsequence $j_l = \max(m_l,j_{l-1}+1)$ and set $\beta_{j_l} = \dfrac{\alpha_{j_l}}{\widetilde\beta_l},\,\gamma'_{j_l} = \dfrac{\alpha_{j_l}}{\widetilde\gamma'_l}$. Then one has, 
$$
\gamma'^b_{j_l} = \frac{\alpha_{j_l}}{\widetilde\gamma'_l} \leqslant \frac{\delta'_{j_l}}{l};
$$
$$
\int\limits_{B_{\delta_{j_l}}(c_{j_l})\setminus B_{\gamma'_{j_l}}(c_{j_l})} V_{N_{j_l},k}\,dv_g= \widetilde\mu_{N_m,k}(\widetilde B_{\widetilde \gamma'_l}(S)) = \tau_S+o(1);
$$
$$
\int\limits_{B_{\gamma'_{j_l}}(c_{j_l})\setminus B_{\beta_{j_l}}(c_{j_l})} V_{N_{j_l},k}\,dv_g=\widetilde\mu_{N_{j_l},k}(\widetilde B_{\widetilde\beta_l}(S)\setminus \widetilde B_{\widetilde \gamma'_l}(S)) = o(1).
$$
Renaming $\{j_l\}$ to $\{m\}$ completes the proof of the lemma.
\end{proof}
Set $\gamma_m = \sqrt{\gamma'_m\beta_m}$. We illustrate the construction presented above by Figure 2.
\begin{figure}
	\includegraphics[width=0.5\textwidth]{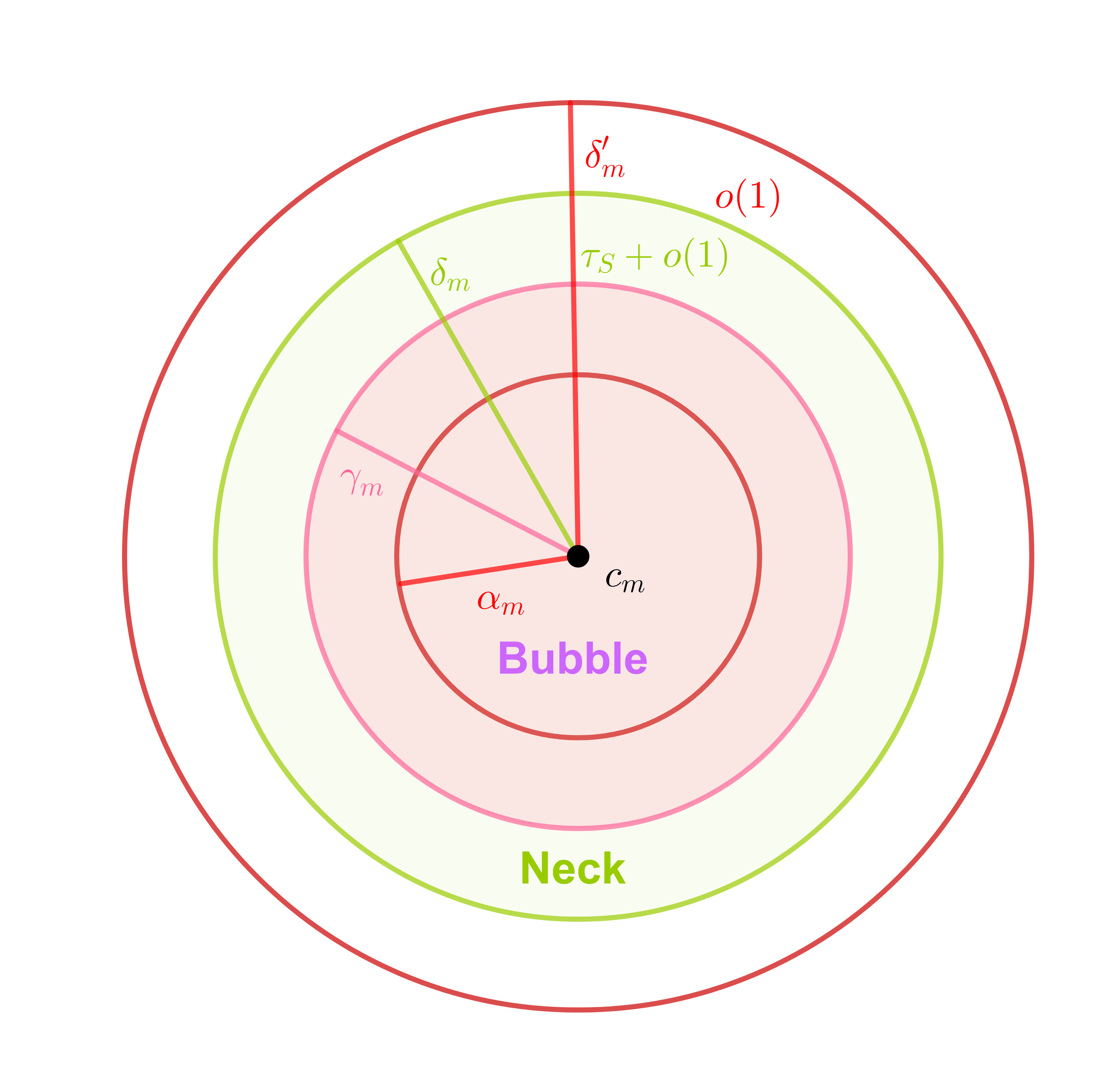}
	\caption{Decomposition of the disk $B_{\delta_m'}(c_m)$. The neck region (shown in green) is collapsing into the south pole as $m\to\infty$ and creates the mass $\tau_S$ there. The bubble region
$B_{\gamma_m}(c_m)$ is shown in pink.}
\end{figure}



Next, we study the regularity of the measure $\widetilde\mu$. It turns out that there are two cases depending on the behaviour of the quantity $\alpha_m^2N_m$.
Fix an open subset $U\Subset\mathbb{S}^2\setminus\{S\}$. We claim that if 
\begin{equation}
\label{bubble:condition}
\alpha_m^2 N_m \to \infty,
\end{equation}
then up to a choice of a subsequence, an analogue of  Theorem~\ref{thm:summary} holds for the restrictions of $\tilde V_{m,k}$ and $\tilde\phi_{m,k}$ to $U$. 
\begin{proposition}
\label{analog}
Assume that the bubble $p$ is such that condition \eqref{bubble:condition} holds for  a strictly increasing subsequence $N_m$, $m=1,2,\dots$.  Then for any given open set  $U\Subset\mathbb{S}^2\setminus\{S\}$  one has:
\begin{itemize}
\item[(1)] $\Delta_{g_{\mathbb{S}^2}} \tilde \phi_{N_m, k} = \Lambda_k^N \, \tilde V_{N_m, k}\, \tilde \phi_{N_m,k}$.
\item[(2)] The $(k+1)$-st Dirichlet eigenvalue of the Schr\"odinger operator
 \begin{equation}
\label{Schrodop}
\Delta_{g_{\mathbb{S}^2}} - \Lambda_k^{N_m}\tilde V_{N_m,k}
\end{equation}
 on $U$  is non-negative.
\item[(3)] $||\tilde V_{N_m,k}||_{L^\infty}\leqslant C{N_m}$, $||\tilde V_{N_m,k}||_{L^1} \le  1$.
\item[(4)] There exists a weak limit $\tilde \phi_{N_m, k} \rightharpoonup \tilde \phi_{k}$
in $H^1(U)$ and $\tilde \phi_{N_m, k} \to \tilde \phi_{k}$ in $L^2(U)$.
\item[(5)] $|\tilde \phi_{N_m, k}|\leqslant 1$ and $|\tilde \phi_k| = 1$ $dv_{g_{\mathbb{S}^2}}$-a.e.
\item[(6)] $\tilde V_{N_m,k}\, dv_{g_{\mathbb{S}^2}}\rightharpoonup^*d\tilde \mu_U$ for some probability measure $d\tilde\mu_U$ on $U$.
\end{itemize}
\end{proposition}
\begin{proof}
The first property is simply \eqref{deltatilde}. Property (2) follows from the fact that the operator \eqref{Schrodop} on $U$ is unitary equivalent to the Schr\"odinger operator $\Delta - \Lambda_k^{N_m}V_{N_m,k}$ on $B_m$. Since the $(k+1)$-st eigenvalue of the latter operator on $M$ is zero, by the Dirichlet bracketing the $(k+1)$-st Dirichlet eigenvalue of \eqref{Schrodop} is non-negative.

The map $R_m$ introduces the conformal factor
\begin{equation}
\label{conffactor}
 g_{\mathbb{S}^2}=\left(\frac{\alpha_m}{\alpha_m^2+(x-c_m)^2}\right)^2  g_0,
\end{equation}
where $g_0$ is the flat metric, locally conformal to $g$.
Therefore,
$$
\tilde V_{N_m, k}=\left(\frac{\alpha_m}{\alpha_m^2+(x-c_m)^2}\right)^{-2} V_{N_m,k}.
$$
At the same time, since $U$ is a compact set away from the south pole, 
$$
\frac{(x-c_m)}{\alpha_m}<C_{U}.
$$
Therefore, 
\begin{equation}
\label{potential:ineq}
|\tilde V_{N_m,k}|\le C \alpha_m^2 V_{N_m,k},
\end{equation}
 and property (3) follows immediately from the analogous property in Theorem \ref{thm:summary}. The same is true about property (4). In property (5), the only condition to check is that   $|\tilde \phi_k| = 1$ holds almost everywhere in the new measure, i.e. $dv_{g_{\mathbb{S}^2}}$-a.e.
Indeed, the conformal factor \eqref{conffactor} satisfies the following bound:
$$
\left(\frac{\alpha_m}{\alpha_m^2+(x-c_m)^2}\right)^2\leqslant \alpha_m^{-2}.$$
Recall the definition \eqref{Enk} of the set $E_{N_m,k}$.  This set has  the property that $|\phi_k|=1$ on its complement,  and by \eqref{Enk:bound} one has that $dv_{g_0}(E_{N_m,k})\leq CN_m^{-1}$

 Therefore, $dv_{g_{\mathbb{S}^2}}(U\cap R_m(B_m\cap E_{N_m}))\leqslant C(N_m \alpha^2_m)^{-1}$, which tends to $0$ by \eqref{bubble:condition}.

Finally, property (6) easily follows from the compactness of the space of measures.
\end{proof}
We claim that Proposition~\ref{analog} allows us to apply the regularity results of Section~\ref{regularity:sec} to the measure $\widetilde \mu_U$. Indeed, the definitions of  good and bad points are purely local as are  the proofs of Propositions~\ref{regularity},~\ref{weaksoleq} and~\ref{dmuk}. The only statement that is not immediate is Proposition~\ref{bad}. However, its proof can be easily modified to make use of assertion (2) of Proposition~\ref{analog}. 

Thus, we can choose a subsequence such that $\widetilde\mu_{N_m,k}|_U\rightharpoonup^*\widetilde\mu_U$, where $\widetilde\mu_U$ is regular outside a finite collection of points. Picking a diagonal subsequence over an exhaustion of $\mathbb{S}^2\setminus S$, we have that 
$\widetilde\mu = \tilde\mu^r + \sum_j \widetilde w_j\delta_{\widetilde q_j} + \tau_S\delta_S$, $\widetilde\mu^r = \widetilde V_\infty\,dv_{g_{\mathbb{S}^2}}$ is a regular measure whose density is the energy density of a harmonic map to a sphere, i.e. $\widetilde V_\infty\in C^\infty(\mathbb{S}^2)$. We call $\tilde q_j$ secondary bubble points. Note that there are at most $k+1$ secondary bubbles.

We continue this procedure inductively at secondary bubbles $\tilde q_j$ until one of the two things happen, either the weight $\tilde w_j<C_R$ or the condition~\eqref{bubble:condition} fails to hold. 
In the former case we call $\tilde q_j$ a {\em terminal bubble}. The following lemma guarantees that this process terminates after finitely many steps.

\begin{lemma}
\label{finiteness:lemma}
One has $\tau_S\leqslant C_R$, and  the  $\widetilde\mu$-mass of the closed southern hemisphere is exactly $C_R$ unless there are secondary bubbles on the equator.
Furthermore,
all secondary bubbles have mass at most $\max(C_R,w-C_R)$.
\end{lemma}
\begin{proof}
By the construction of $\alpha_m$, the mass of the open southern hemisphere is at most $C_R$. Therefore, $\tau_S\leqslant C_R$ and the mass of the closed southern hemisphere is exactly $C_R$ unless there exists a  
secondary bubble $q$ on the equator. Assume that its mass $w_q$ is strictly greater than $\max(C_R,w-C_R)\geqslant w-C_R$. Let $d_m\in\partial B_{\alpha_m}(c_m)$ be such that $q= R_m(d_m)$. By Lemma~\ref{tree:lemma1} one has $|d_m|\leqslant |c_m| + \alpha_m\ll\varepsilon_m$, i.e. $d_m\in B_{\varepsilon_m}(p)$. Let $U = R_m(B_{\alpha_m/3}(d_m))$ be a fixed neighbourhood $q$. Since $w_q>w-C_R$ for large $m$ one has 
$$
\widetilde \mu_{N_m,k}(\mathbb{S}^2\setminus U)=\mu_{N_m,k}(B_m\setminus B_{\alpha_m/3}(d_m))<C_R.
$$   
Hence, $\alpha(d_m)<\frac{\alpha_m}{3}$ which contradicts the definition of $\alpha_m$.

Assume that there is a secondary bubble of mass strictly greater than $\max(C_R,w-C_R)$ somewhere. Then it can not be be in the open southern hemisphere since its mass is at most $C_R$. The previous argument shows that it can not be on the equator. Thus, it is in the open northern hemisphere. But the mass of the closed southern hemisphere is at least $C_R$, so we obtain a contradiction with the the fact that the total mass of the bubble is equal to $w$. \end{proof}

Let us now  assume that the initial bubble $p$ does not satisfy the condition~\eqref{bubble:condition}, i.e. up to a choice of a subsequence $\alpha_m^2N_m = O(1)$. In this case by inequality~\eqref{potential:ineq} the potentials $\widetilde V_{N_m,k}$ are uniformly bounded on any given open set $U\Subset \mathbb{S}^2\setminus \{S\}$. Therefore, once again one could choose a diagonal subsequence and imply that there exists a $*$-weak limit $\widetilde V_\infty\in L^1(\mathbb{S}^2)$ such that $\widetilde V_{N_m,k}\,dv_{g_{\mathbb{S}^2}}\rightharpoonup^*\widetilde V_\infty\,dv_{g_{\mathbb{S}^2}} + \tau_S\delta_S$, where $\tau_S\leqslant C_R$ and $\widetilde V_\infty\in L^\infty(U)$ for all $U\Subset \mathbb{S}^2\setminus \{S\}$. In particular, the bubble tree construction stops at such bubbles since there are no secondary bubbles, only a possible mass concentration  near the south pole. 
We see that at any non-terminal bubble (regardless the behaviour of $\alpha^2_m N_m$) the measure $\widetilde \mu$ is regular up to possible concentration at finitely many atoms. 
\begin{remark} 
Note that the argument above takes two different routes depending on whether the condition  \eqref{bubble:condition} is satisfied. This condition provides a relation between  the rescaling $\alpha_m$ and  the 
blow-up rate of the maximizing  subsequence given by $N_m$.  A dichotomy of  this kind  appears to be intrinsic to the problem, as a similar issue arises in the bubble tree construction in \cite[Section 5]{Pet2}.
\end{remark} 
Let us now describe the construction of the bubble tree. The root of the tree is the surface $M$, and its direct descendants are the atoms $p_i$. As described above, each atom gives rise to bubbles, and each bubble, after appropriate rescaling may give rise to secondary bubble points, and so on. Each branch of the tree stops at a terminal bubble, and in view of Lemma \ref{finiteness:lemma} the bubble tree is finite. We summarize its properties in the following theorem.
\begin{theorem}[Bubble tree]
\label{tree:summaryI}
For any non-terminal bubble $b$ there exists a point $p^b\in M$, a sequence of points $c^b_m\to p^b$ and a sequence of scales 
$$
\alpha^b_m\ll\beta^b_m\ll\gamma^b_m\ll\gamma'^b_m\ll\delta'^b_m\ll\delta^b_m\ll\varepsilon^b_m\ll 1
$$
and for any terminal bubble $b$ there exists a sequence of sequence of points $c^b_m\to p^b$ and a sequence of scales
$$
\delta'^b_m\ll\delta^b_m\ll\varepsilon^b_m\ll 1
$$
such that 
\begin{itemize}
\item[1)] Any two bubbles $b_1, b_2$ are either away from one another or one of them is  a descendent of the other. 
In the former case one has  that the intersection $B_{\varepsilon_m^{b_1}}(c_m^{b_1})\cap B_{\varepsilon_m^{b_2}}(c_m^{b_2})$ is empty.
In the latter case, $b_1$ is secondary to $b_2$ or $b_1\prec b_2$ if $p^{b_1} = p^{b_2}$, $\varepsilon^{b_1}_m\ll\alpha^{b_2}_m$ and $B_{\varepsilon^{b_1}_m}(c^{b_1}_m)\subset B_{\beta^{b_2}_m}(c^{b_2}_m)$.

\item[2)] As $m\to \infty$, the following asymptotic relations hold: 
$$
\mu_{m,k}(B_{\gamma'^b_m}(c^b_m)\setminus B_{\beta^b_m}(c^b_m)\cup B_{\varepsilon^b_m}(c^b_m)\setminus B_{\delta'^b_m}(c^b_m)) = o(1),
$$
$$
\mu_{m,k}(B_{\delta'^b_m}(c^b_m)\setminus B_{\gamma'^b_m}(c^b_m)) = \tau^b_S+o(1),
$$
where $\tau^b_S\leqslant C_R$.

\item[3)] Let $M_m = M\setminus \bigcup_b B_{\delta^b_m}(c^b_m)$. Then
$$
\mu_{m,k}|_{M_m}\rightharpoonup^* V_\infty dv_g,
$$
where $V_\infty\in C^\infty(M)$.  We will refer to $M_m$ as the {\it regular region}.

\item[4)] Set $R^b_m(x) = \pi((\alpha^b_m)^{-1}(x-c^b_m))$ and define the {\it bubble region} $\B_m(b) = B_{\gamma^b_m(c^b_m)}\setminus \bigcup_{a\prec b} B_{\delta^a_m}(c^a_m)$. Then
\begin{equation}
\label{vinfty}
(R^b_m)_*\mu_{m,k}|_{B_m(b)}\rightharpoonup^* \widetilde V^b_\infty dv_{g_{\mathbb{S}^2}} ,
\end{equation}
where $\widetilde V^b_\infty\in L^1(\mathbb{S}^2)$ and for any $U\Subset \mathbb{S}^2\setminus\{S\}$ one has $\widetilde V^b_\infty|_U\in L^\infty(U)$.
\end{itemize}
\end{theorem}  

\begin{definition}
We say that the bubble $b$ is {\em of type I} if $\widetilde V^b_\infty\in L^\infty(\mathbb{S}^2)$. Otherwise, we say that the bubble $b$ is {\em of type II}.
\end{definition}

\begin{remark}
\label{rem:essential}
 The real difference between type I and type II bubbles arises if  $\widetilde V^b_\infty\notin L^p(\mathbb{S}^2)$ for any $p>1$, since in this case one can not guarantee that the spectrum of the 
corresponding Laplacian is discrete (see subsection \ref{subsec:Lp} and \cite[Section 2.3]{Kok2} for details).
\end{remark}
For Type II bubbles one needs to modify the scales obtained in Proposition~\ref{tree:summaryI} in the way described below, see also Figure 3 for an illustration.

\begin{figure}[htb]
\begin{center}
\includegraphics[width=0.5\textwidth]{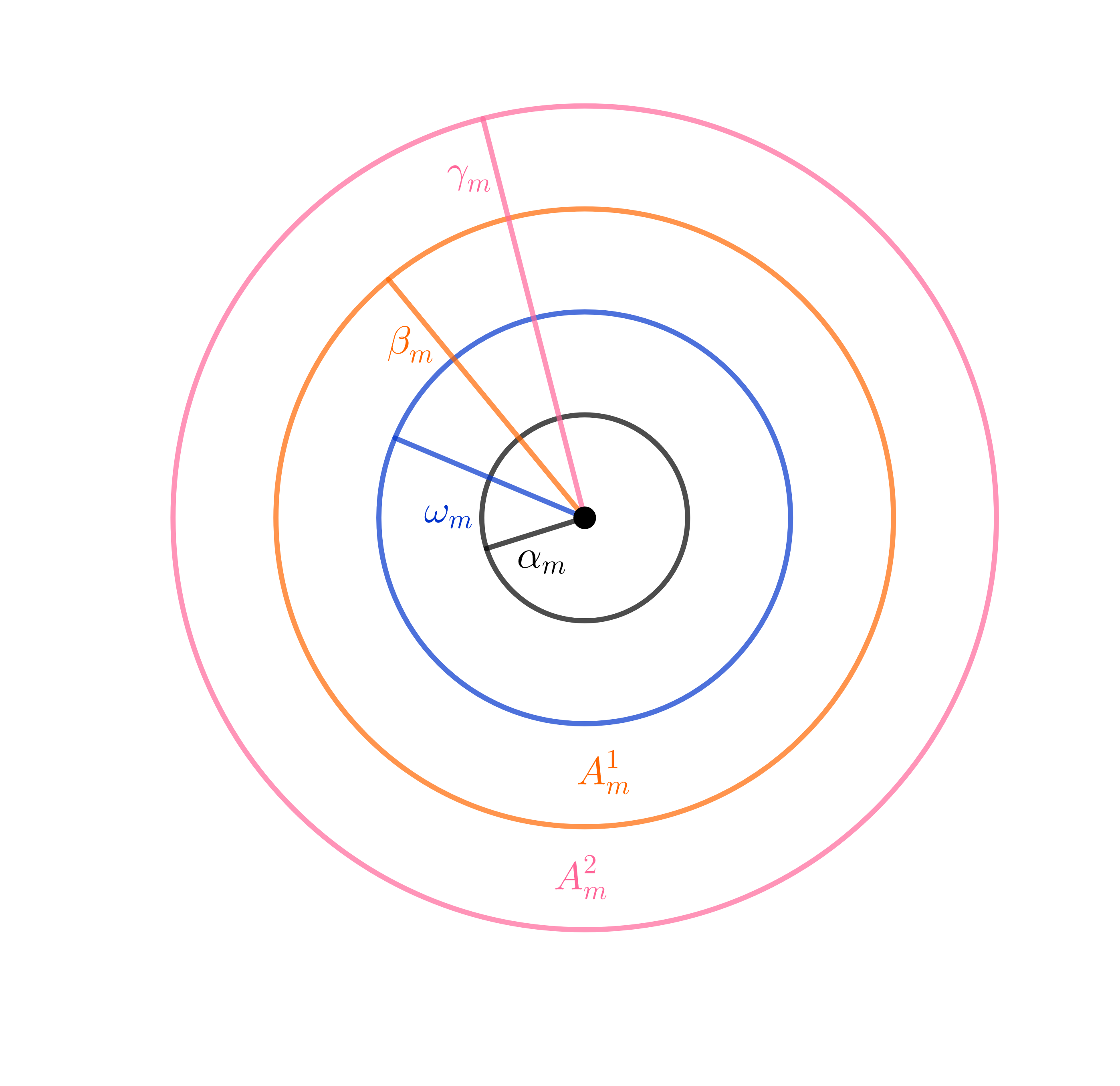}
\end{center}
\caption{Fine structure of a type II bubble. The outer ring $A_m^2$ contains an amount of mass bounded below  by~\eqref{bubbleII:condition}.}
\end{figure} 
\begin{proposition}
Let $b$ be a bubble of type II. Then one can redefine $\beta_m^b$ and additionally find $\omega_m^b$ such that 
$$
\alpha_m^b\ll\omega^b_m\ll\beta^b_m\ll \gamma^b_m
$$
and the following holds.
Define $A^1_m(b) = B_{\beta^b_m}(c_m^b)\setminus B_{\omega^b_m}(c^b_m)$, which is a part of the bubble region, and a ``collar region'' $A^2_m(b) = B_{\gamma^b_m}(c^b_m)\setminus B_{\beta^b_m}(c^b_m)$. Then
\begin{equation}
\label{A12}
\mu_{m,k}(A^i_m(b)) = o(1)
\end{equation}
for $i=1,2$ as $m\to \infty$, and 
\begin{multline}
\label{bubbleII:condition}
\mu_{m,k}(A^2_m(b)) \gg \sum_{a}\left(\frac{1}{\sqrt{\ln\frac{\gamma'^a_m}{\gamma^a_m}}} + \frac{1}{\sqrt{\ln\frac{\delta^a_m}{\delta'^a_m}}} + \frac{1}{\sqrt{\ln\frac{\varepsilon^a_m}{\delta^a_m}}}\right) +\\ \sum_{\text{$a$ is of type I}}\frac{1}{\sqrt{\ln\frac{\gamma^a_m}{\beta^a_m}}}.
\end{multline}
\end{proposition}
\begin{proof} 

Let $\widetilde B_r(S)$ be a neighbourhood of the south pole $S\in \mathbb{S}^2$ defined as the complement of $B_{r^{-1}}(N)$, where the distance is measured in the metric $(\pi^{-1})^*g_{\mathbb{R}^2}$. Set 
$$
f(r) = \int\limits_{\widetilde B_r(S)} \widetilde V^b_\infty\,dv_{g_{\mathbb{S}^2}}.
$$ 
Since $\widetilde V^b_\infty\not\in L^\infty(\widetilde B_r(S))$ for any $r$, it is nonzero almost everywhere is some neighborhood of $S$, and therefore $f$ is a non-decreasing function satisfying 
$f(r)>0$ for $r>0$. Fix a small $r_0>0$.
 Let $\nu_m$ be the square root of the r.h.s of~\eqref{bubbleII:condition}, then $\nu_m = o(1)$ and for large enough $m$ one has $\nu_m\leqslant f(r_0)$. For such $m$ we define 
 $$
 \widetilde \beta_m = \min\{r|\, f(r) = \nu_m\}
 $$
and set $\widetilde \omega_m = \sqrt{\widetilde \beta_m}$. Then  $1\gg\widetilde \omega_m\gg\widetilde \beta_m$. Set $\widetilde \gamma_m = \frac{\alpha^b_m}{\gamma^b_m}$ so that $\partial \widetilde B_{\widetilde\gamma_m}(S) = \partial R_m^b(B_{\gamma^b_m}(c^b_m))$. For a fixed $n$ there exists $m_n$ such that for all $m\geqslant m_n$ one has the following,
$$
n\widetilde\gamma_m\leqslant \widetilde\beta_n;
$$
$$
\widetilde \mu_{N_m,k}(\widetilde B_{\widetilde\beta_n}(S)\setminus \widetilde B_{\widetilde\gamma_m}(S))\geqslant \frac{1}{2}\int\limits_{\widetilde B_{\widetilde\beta_n}(S)} \widetilde V^b_\infty\,dv_{g_{\mathbb{S}^2}} = \frac{1}{2}\nu_n\geqslant \nu_m;
$$
$$
\widetilde \mu_{N_m,k}(\widetilde B_{\widetilde\omega_n}(S)\setminus \widetilde B_{\widetilde\gamma_m}(S))\leqslant 2\int\limits_{\widetilde B_{\widetilde\omega_n}(S)} \widetilde V^b_\infty\,dv_{g_{\mathbb{S}^2}}
\xrightarrow{n\to\infty} 0
$$

Define a subsequence $j_l = \max(m_l,j_{l-1}+1)$ and set $\beta^b_{j_l} = \dfrac{\alpha^b_{j_l}}{\widetilde \beta_l}$, $\omega^b_{j_l} = \dfrac{\alpha^b_{j_l}}{\widetilde \omega_l}$. After elementary calculations, the previous inequalities become
\begin{equation}
\label{aux1}
l\beta^b_{j_l}\leqslant \gamma^b_{j_l}
\end{equation}
\begin{equation}
\label{aux2}
\mu_{j_l,k}(A_{j_l}^2(b))\geqslant l_{j_l}
\end{equation}
\begin{equation}
\label{aux3}
\mu_{j_l,k}(A_{j_l}^1(b)\cup A_{j_l}^2(b))\xrightarrow{l\to\infty} 0
\end{equation}
Let us rename the subsequence $\{j_l\}$ to $\{m\}$. Then \eqref{aux1} implies that $\gamma_m~\gg~\beta_m$, \eqref{aux2} implies \eqref{bubbleII:condition} and \eqref{aux3} implies
\eqref{A12}. This completes the proof of the proposition.
\end{proof}
\subsection{Construction of test-functions}
\label{subsec:test}
In this section we describe the test-space for $\lambda_i(V_{N_m,k})$.  Let us introduce some notation. In addition to the bubble region $\B_m(b)$ and the regular region $M_m$ introduced in Theorem \ref{tree:summaryI},
we set $M_m'= M\setminus \bigcup_b B_{\varepsilon^b_m}(c^b_m)$, $\B_m'(b)=B_{\beta^b_m}(c^b_m)\setminus\bigcup_{a\prec b} B_{\varepsilon^a_m}(c^a_m)$, and introduce the {\it neck regions} $\A_m(b)= B_{\delta^b_m}(c^b_m)\setminus B_{\gamma^b_m}(c^b_m)$, $\A'_m(b) := B_{\delta'^b_m}(c^b_m)\setminus B_{\gamma'^b_m}(c^b_m)$. For terminal bubbles we set $\A_m(b) = \B_m(b) = B_{\delta^b_m}(c_m^b)$ and $\A'_m(b) =\ B'_m(b) = B_{\delta'^b_m}(c_m^b)$.

First, we construct test-functions supported in $M_m$. For that we take the eigenfunctions for $(M,V_\infty)$ and multiply them by a logarithmic cut-off function $\rho^M_m\in C_0^\infty(M_m)$ equal to $1$ on $M_m'$. We denote such a space of test-functions constructed from the first $j$ eigenfunctions (including constants) by $F_j^M$.

Similarly, we define test-functions supported in the bubble region $\B_m(b)$ for a type I bubble $b$. We take the eigenfunctions for $(\mathbb{S}^2, V^b_\infty)$, transplant them to $M$ and multiply them by a logarithmic cut-off function $\rho^b_m\in C^\infty_0(\B_m(b))$ equal to $1$ on $\B_m'(b)$. We denote such a space of test-functions constructed from the first $j$ eigenfunctions (including constants) by $F_j^b$.

For each terminal bubble $b$ we simply use the logarithmic cut-off function $\rho_m^b\in C_0^\infty(\A_m(b))$ equal to $1$ on $\A_m'(b)$. Similarly, for each neck region with non-zero mass on any bubble $b$ we use the logarithmic cut-off $\tau_m^b\in C_0^\infty(\A_m(b))$ equal to $1$ on $\A_m'(b)$. We denote the space spanned by these functions by $F_{neck}$. Note that $\dim F_{neck}$ is equal to the number $t$ of terminal bubbles and necks of non-zero mass.

The situation for type II bubbles is more complicated. In particular, the test-functions associated with type II bubbles are not supported on that bubble, but rather equal constant outside the bubble. Let $b$ a type II bubble. 
First of all we modify the potential $V_{N_m,k}$ to be equal to $0$ on $A^1_m(b)$. This only increases the eigenvalues and does not change the behaviour as $m\to \infty$ since $\mu_{N_m,k}(A^1_m(b)) \to 0$.

Set $\B''_m(b) = B_{\omega_m^b}(c^b_m)$. Let $\widetilde V^b_m$ be a potential on $\mathbb{S}^2$ defined by 
\begin{equation}
\label{type2pot}
\widetilde V^b_m\,dv_{\mathbb{S}^2} = 
\begin{cases}
(R^b_m)_*(V_{N_m,k}\,dv_g) &\text{on $R_m^b(\B''_m(b))$}\\
0&\text{otherwise}
\end{cases}
\end{equation}
Let $\psi^b_m$ be a a linear combination of  the eigenfunctions of  $(\mathbb{S}^2,\widetilde V^b_m)$. In particular, $\psi^b_m$ is harmonic on the complement to $R_m^b(B''_m(b))$. Let $\theta^b_m\in C_0^\infty (R^b_m( B_{\frac{1}{2}\beta_m}(c^b_m)))$ be a cut-off function equal to $1$ on $R^b_m(B_{2\omega^b_m}(c^b_m))$. Define
\begin{equation}
\label{eq:psitilde}
\widetilde \psi^b_{m} = (\psi^b_{m} - \psi^b_{m}(S))\theta^b_{m} + \psi^b_{m}(S),
\end{equation}
where $S$ is the South pole.

The following proposition shows that the Rayleigh quotients of  the functions $\widetilde{\psi}^b_m$ and $\psi^b_m$ are close as $m\to \infty$. 
\begin{proposition}
\label{claim:ineq}
As $m\to \infty$, we have the inequality
\begin{equation}
\label{tilde:ineq}
\int_{\mathbb{S}^2}|\nabla \widetilde \psi^b_{m}|^2 \leqslant (1+o(1)) \int_{\mathbb{S}^2}|\nabla \psi^b_{m}|^2.
\end{equation}
Also, for any $m$,
\begin{equation}
\label{claim:eq}
\int_{\mathbb{S}^2} (\widetilde \psi^b_{m})^2\widetilde V^b_{m}dv_{g_{\mathbb{S}^2}} = \int_{\mathbb{S}^2} (\psi^b_{m})^2\widetilde V^b_{m}dv_{g_{\mathbb{S}^2}}.
\end{equation}
\end{proposition}
\begin{proof}
The equality \eqref{claim:eq} is immediate, since $\widetilde \psi^b_{m} = \psi^b_{m}$ on the support of $\widetilde V^b_{m}$. To prove the inequality we note that 
\begin{equation}
\label{eq:u}
u = \psi^b_{m} - \psi^b_{m}(S)
\end{equation}
 is harmonic on the annulus $R^b_m(A^1_m(b))$. The nodal line of $u$ passes through $S$. By maximum principle the nodal set can not contain a closed arc outside $R^b_m(\B''_m(b))$, therefore the nodal set goes all the way from $S$ to $R^b_m(\partial \B''_m(b))$. We will use the following lemma.
\begin{lemma}
\label{sup:lemma}
There exists a universal constant $C$ such that
$$
||u||^2_{L^\infty( \widetilde{A}^1_m(b)  )}\leqslant C \int_{A^1_m(b)} |\nabla u|^2,
$$
where $\widetilde{A}^1_m(b) =B_{2\omega^b_m}(c^b_m)\setminus B_{\frac{1}{2}\beta_m}(c^b_m) \subset  A^1_m(b)$.
\end{lemma}
\noindent{\it Proof of Lemma \ref{sup:lemma}.}
Let $x$ be a point where $|u|$ achieves the maximum on $\widetilde{A}^1_m(b)$. Assume for simplicity of notations that the coordinates are chosen in such a way that  $c^b_m=0$. Then the nodal line of $u$ intersects both boundary components of the annulus $B_{2|x|}(0)\setminus B_{\frac{1}{2}|x|}(0)$. We will show that 
\begin{equation}
\label{sup:ineq}
|u(x)|^2 \leqslant C \int_{B_{2|x|}(0)\setminus B_{\frac{1}{2}|x|}(0)} |\nabla u|^2,
\end{equation}
which obviously implies the required inequality.

Note that both sides of the inequality~\eqref{sup:ineq} are scale-invariant. Therefore, without loss of generality, we may assume that $|x|=2$. Since the nodal line intersects both boundary components, one has that  
\begin{equation}
\label{nodallemma}
\int_{B_4(0)\setminus B_1(0)} u^2\leqslant C\int_{B_4(0)\setminus B_1(0)} |\nabla u|^2.
\end{equation}
Indeed, for each $\rho\in[1,4]$ let $x_\rho\in \partial B_\rho$ be a point such that $u(x_\rho) = 0$. Let $s$ be a natural parameter along $\partial B_\rho$, i.e. $s = \frac{\phi}{2\pi\rho}$ in polar coordinates. Then one has
\begin{equation*}
\begin{split}
\int_{S_\rho} u^2 =& \int_{S_\rho}\left(\int_{x_\rho}^s \partial_su(\rho,s)\,ds\right)^2\leqslant \int_{S_\rho}\left(2\pi\rho\int_{S_\rho} |\partial_su(\rho,s)|^2\,ds\right)\leqslant \\
&(2\pi\rho)^2\int_{S_\rho}|\nabla u|^2.
\end{split}
\end{equation*}
Therefore, one has on $A= B_4\setminus B_1$ that
$$
\int_{A} u^2 = \int_1^4\left(\int_{S_\rho}u^2\right)d\rho\leqslant \int_1^4\left((2\pi\rho)^2\int_{S_\rho}|\nabla u|^2\right)d\rho\leqslant 64\pi^2\int_A|\nabla u|^2,
$$
which proves inequality \eqref{nodallemma}.

As was mentioned above, $u$ is harmonic in $A_m^1(b)$. Therefore, $u^2$ is subharmonic, and hence
$$
u^2(x)\leqslant \frac{1}{\pi}\int_{B_1(x)} u^2\leqslant \frac{1}{\pi}\int_{B_4(0)\setminus B_1(0)} u^2\leqslant C\int_{B_4(0)\setminus B_1(0)} |\nabla u|^2.
$$
Here in the first inequality we used the mean value theorem, and in the second inequality the inclusion $B_1(x) \subset B_4(0)\setminus B_1(0)$, which follows from the normalization $|x|=2$. 
This completes the proof of Lemma \ref{sup:lemma}. \qed

\smallskip

Let us continue with the proof of Proposition \ref{claim:ineq}. Let $\alpha_m>0$ be a number to be chosen later. Combining the Cauchy-Schwarz inequality with the arithmetic-geometric mean inequality and using \eqref{eq:psitilde}, we obtain for any $\alpha_m>0$: 
\begin{equation}
\label{eq:trick}
\int_{\mathbb{S}^2}|\nabla \widetilde \psi^b_{m}|^2\leqslant (1+\alpha_m)\int |\nabla \psi^b_{m}|^2(\theta^b_m)^2 +\left (1+\frac{1}{\alpha_m}\right)\int u^2|\nabla\theta^b_m|^2,
\end{equation}
where $u$ is defined by \eqref{eq:u}.
Note  that by construction $\operatorname{supp} (\nabla \theta_m^b) \subset \widetilde{A}_m^1(b)$. Therefore, using   Lemma~\ref{sup:lemma} to estimate the second term and taking into account that 
$\theta^b_m\leqslant 1$ to estimate the first one, we get:
$$
\int_{\mathbb{S}^2}|\nabla \widetilde \psi^b_{m}|^2\leqslant \left(1+a_m + C\left(1+\frac{1}{\alpha_m}\right)\int|\nabla\theta^b_m|^2\right)\int |\nabla \psi^b_{m}|^2.
$$ 
Setting $\alpha_m= \left(\int|\nabla\theta^b_m|^2\right)^{1/2}$, and noting that with this choice $\alpha_m=o(1)$ by Section~\ref{cutoff:section},  completes the proof of Proposition \ref{claim:ineq}.
\end{proof}
Let us now define  the space of test-functions associated with a type II bubble. We denote  by $E^b_j$ the space of test-functions $\widetilde{\psi}^b_m$ constructed from the functions 
$\psi^b_m$ which are represented as linear combinations of the first $j$ eigenfunctions orthogonal to constants. Note that by our construction,  if one takes the constant function on the type II bubble, then it yields a constant function on $M$, i.e. constant test-functions on different type II bubbles yield the same test-function on $M$. To compensate for that we need to add $(s-1)$ functions, where $s$ is the number of type II bubbles. For $(s-1)$ of those bubbles we add a logarithmic cut-off $\rho^b_m\in C^\infty_0(B_{\gamma'^b_m}(c_m^b))$ which is equal to $1$ on $B_{\gamma^b_m}(c_m^b)$. We denote by $\widetilde a$ the remaining type II bubble and by $E$ the $s$-dimensional space spanned by $1$ and these $(s-1)$ functions.
\subsection{Eigenvalue bounds}
\label{subsec:eigbounds}
In the notation of the previous subsection, let 
$$
F = F^M_{j+1}\bigoplus\limits_{\text{$b$ of type I}} F^b_{j_b+1} \bigoplus\limits_{\text{$a$ of type II}} E^a_{j_a}\oplus F_{neck}\oplus E 
$$
for some fixed natural numbers $j,j_a$ and $j_b$, where the index $b$ runs over all bubbles of type I, and the index $a$ runs over all bubbles of type II.
\begin{proposition}
\label{separation:proposition}
For any given given natural numbers $k, j, j_a,j_b$, where the index $b$ runs over al bubbles of type I, and the index $a$ runs over all bubbles of type II, one has as $m\to \infty$:
$$
\lambda_{\dim F-1}(M,V_{N_m,k})\leqslant \max\limits_{\text{$b$ of type I, $a$ of type II}}\{ \lambda_j(M,V_\infty),\lambda_{j_b}(\mathbb{S}^2, \widetilde V^b_{\infty}),\lambda_{j_a}(\mathbb{S}^2,\widetilde V^a_m)\} + o(1).
$$
\end{proposition}
\begin{proof}
Let $u\in F$. Then there exists a constant $D_m$ such that for any bubble $b$ of type $I$
 \begin{equation}
u_m = 
\begin{cases}
\psi^M&\text {on $M_m'$}\\
\rho^M_m\psi^M + (1-\rho^M_m)D_m&\text{on $M_m\setminus M'_m$}\\
\psi^b&\text{on $\B'_m(b)$}\\
\rho^b_m\psi^b + (1-\rho^b_m)D_m&\text{on $\B_m(b)\setminus \B_m'(b)$},\\
\end{cases}
\end{equation}
where $\psi^M$ and $\psi^b$ are linear combinations of  the first $j+1$ eigenfunctions of $(M,V_\infty)$ and the first $j_b+1$ eigenfunctions of  $(\mathbb{S}^2,\widetilde V^b_\infty)$,  respectively (in both cases, the constants are included). Furthermore, for any bubble $b$ of type I with neck of non-zero mass  or a terminal bubble one has
 \begin{equation}
 u_m = 
\begin{cases}
C^b&\text {on $\A_m'(b)$}\\
\tau^b_m C^b + (1-\tau^b_m)D_m&\text{on $\A_m(b)\setminus\A'_m(b)$},\\
\end{cases}
\end{equation}
where $C^b$ are some constants. If the neck mass is zero, then $u = D_m$ on $\A_m(b)$. Finally, for type II bubbles $a$ with non-zero mass neck one has
 \begin{equation}
 u_m = 
\begin{cases}
\widetilde\psi^a_m&\text {on $\B_m'(a)$}\\
D^a_m&\text{on $A^2_m(a)$}\\
C^a\rho_m^a + D^a_m(1-\rho_m^a)&\text{on $B_{\gamma'^a_m}(c^a_m)\setminus B_{\gamma^a_m}(c^a_m)$}\\
C^a&\text{on $B_{\delta'^a_m}(c_m^a)\setminus B_{\gamma'^a_m}(c^a_m)$}\\
C^a\tau^a_m + D_m(1-\tau^a_m)&\text{on $B_{\delta^a_m}(c_m^a)\setminus B_{\delta'^a_m}(c_m^a)$,} 
\end{cases}
\end{equation}
where $D^{\widetilde a}_m = D_m$ and $\widetilde{\psi}^a_m$ is obtained by the cut-off construction from a linear combination $\psi^a_m$ of the first $j_a+1$ eigenfunctions of $(\mathbb{S}^2,\widetilde V^a_m)$ defined by \eqref{eq:psitilde}. Finally, for type II bubble $a$ with zero mass neck one has
 \begin{equation}
 u_m = 
\begin{cases}
\widetilde\psi^a_m&\text {on $\B_m'(a)$}\\
D^a_m&\text{on $A^2_m(a)$}\\
D_m\rho_m^a + D^a_m(1-\rho_m^a)&\text{on $B_{\delta^a_m}(c^a_m)\setminus B_{\gamma^a_m}(c^a_m).$}\\
\end{cases}
\end{equation}

We are now ready to estimate the Rayleigh quotient of $u$. We do it step by step. 

{\bf 1. On $M_m$.} Since the space $F_{j+1}^M$ is finite dimensional, there exists a constant $C_j$ such that for any $x\in M$,
\begin{equation}
\label{eq:dim}
|\psi^M (x)|^2\leqslant C_j \int_M \left(\psi^M\right)^2V_\infty.
\end{equation}
Then we have for any $0\le \delta \le 1$: 
\begin{equation*}
\begin{split}
&\int\limits_{M_m\setminus M'_m} |\nabla u|^2 \leqslant  \\
&\int\limits_{M_m\setminus M'_m}(1+2\delta)(\rho_m^M)^2|\nabla \psi^M|^2 + (1+\delta+\frac{1}{\delta})(\psi^M)^2|\nabla\rho_m^M|^2 + (1 + \frac{2}{\delta})D_m^2|\nabla \rho_m^M|^2\leqslant\\ 
&(1+2\delta)\int\limits_{M_m\setminus M'_m}|\nabla \psi^M|^2 + (1+\frac{2}{\delta})\int_{M_m\setminus M'_m}|\nabla\rho_m^M|^2 C \int_{M} (\psi^M)^2V_\infty + (1 + \frac{2}{\delta})D_m^2\int|\nabla\rho^M_m|^2
\end{split}
\end{equation*}
Here the first inequality follows, similarly to  \eqref{eq:trick}, from the Cauchy-Scwarz inequality combined with the arithmetic-geometric mean inequality. 
Set $\delta = \sqrt{\int|\nabla\rho_m^M|^2} = o(1)$. This yields
\begin{equation*}
\begin{split}
&\int\limits_{M_m}|\nabla u|^2\leqslant (1+o(1)) \int\limits_{M_m}|\nabla \psi^M|^2 + o(1)\int\limits_M (\psi^M)^2V_\infty + (2+o(1))D_m^2\sqrt{\int|\nabla\rho_m^M|^2}  \\
&\leqslant \lambda_{j}(M,V_\infty)(1+o(1))\int\limits_M (\psi^M)^2V_\infty + I^M_m,
\end{split}
\end{equation*}
where
$$
I_m^M = (2+o(1)) D_m^2\sqrt{\int|\nabla\rho^M_m|^2}.
$$
Note that by consctruction the space $F_{j+1}$ is generated by the first $j+1$ eigenfunctions including constants, and hence  the $j$-th nonzero eigenvalue $\lambda_j$ appears on the right-hand side of the inequality.
At the same time, using \eqref{eq:dim} one has
\begin{multline}
\left|\int\limits_M(\psi^M)^2 V_\infty - \int\limits_{M'_m}(\psi^M)^2 V_{N_m,k}\right|\leqslant\\  \left( C_j \int\limits_M(\psi^M)^2 V_\infty \right) \, \left|\int\limits_MV_\infty - \int\limits_{M'_m}V_{N_m,k}\right|\leqslant o(1)\int\limits_M(\psi^M)^2 V_\infty
\end{multline}
Here in the last inequality we used the third assertion of Theorem \ref{tree:summaryI}, as well as the fact that by construction $\mu_{m,k}(M_m\setminus M_m')=o(1)$.
Putting everything together and taking into account that $u|_{M_m'}=\psi^M$, we have that 
$$
\int\limits_{M_m}|\nabla u|^2 \leqslant \lambda_{j}(M,V_\infty)(1+o(1))\int\limits_{M_m} u^2V_{N_m,k} + I_m^M.
$$ 
The term $I_m^M$ will be dealt with later.

{\bf 2. On $B_m(b)$.} The same argument follows through on $B_m(b)$ for type I bubbles $b$. One has
$$
\int\limits_{B_m(b)}|\nabla u|^2 \leqslant \lambda_{j}(\mathbb{S}^2,\widetilde V^b_\infty)(1+o(1))\int\limits_{B_m(b)} u^2V_{N_m,k} + I_m^b,
$$
where
$$
I_m^b = (2+o(1)) D_m^2\sqrt{\int|\nabla\rho^b_m|^2}.
$$

{\bf 3. On $A_m(b)$ for type I bubbles.} On the neck regions $A_m(b)$ of non-zero mass for type I bubbles and terminal bubbles one has
$$
\int\limits_{A_m(b)}|\nabla u|^2\leqslant (1+\delta)(C^b)^2\int|\nabla \tau^b_m|^2 + (1+\frac{1}{\delta})D_m^2\int |\nabla \tau^b_m|^2.
$$
Since 
$$
\int\limits_{A'_m(b)}u^2V_{N_m,k} = (C^b)^2(\tau_S^b + o(1)).
$$
Setting $\delta = \sqrt{\int|\nabla \tau^b_m|^2}$ one has
$$
\int\limits_{A_m(b)}|\nabla u|^2\leqslant o(1)\int\limits_{A_m(b)} u^2V_{N_m,k} + J^b_m,
$$
where 
$$
J^b_m = (2+o(1))D_m^2\sqrt{\int|\nabla \tau^b_m|^2}
$$

{\bf 4. On $A_m(a)$ for type II bubbles.} Similar argument on the neck regions $A_m(a)$ for type II bubbles (this bound holds for both zero and non-zero mass of the neck) yields
$$
\int\limits_{A_m(a)}|\nabla u|^2\leqslant o(1)\int\limits_{A_m(a)} u^2 V_{N_m,k} + J^a_m + K^a_m,
$$
where
$$
J^a_m = (2+o(1))D_m^2\sqrt{\int|\nabla \tau^a_m|^2};\qquad 
K_m^a = (2+o(1)) (D^a_m)^2\sqrt{\int|\nabla\rho^a_m|^2}.
$$

{\bf 5. On $B_m'(a)$ for type II bubbles.} By the construction of test-functions one has
$$
\int\limits_{B_m'(a)} |\nabla u|^2\leqslant (\lambda_{j_a}(\mathbb{S}^2,\widetilde V^a_m) + 1)\int_{B'_m(a)} u^2V_{N_m,k}
$$

{\bf 6. Dealing with $I,\,J,\,K$ terms.} 
We note that by condition~\eqref{bubbleII:condition} and the estimate in Section~\ref{cutoff:section}, one has
$$
\sum_a \left(I^a_m + J^a_m\right) + K_m^{\widetilde a}\ll D_m^2\int\limits_{A^2_m(\widetilde a)} V_{N_m,k} = \int\limits_{A^2_m(\widetilde a)} u^2V_{N_m,k}.
$$
Similarly, for type II bubble $a\ne\widetilde a$ one has
$$
K_m^a\ll (D_m^a)^2\int\limits_{A^2_m(a)} V_{N_m,k} = \int\limits_{A^2_m(a)} u^2V_{N_m,k}.
$$

Summing all these terms together completes the proof of Proposition \ref{separation:proposition} .
\end{proof}

Let $w_M=\int_M V_\infty$ be the area of the regular part of the surface $M$. 
If $w_M\ne 0$, then define $d_M$ by 
$$
\lambda_{d_M}(M,V_\infty)<\Lambda_k(M,\mathcal C)\leqslant \lambda_{d_M+1}(M,V_\infty).
$$
Similarly, for all type I bubbles $b$ we set  $w_b=\int_{\mathbb{S}^2} \widetilde{V}^b_\infty dv_{g_{\mathbb{S}^2}}$, where $ \widetilde{V}^b_\infty$  is defined by \eqref{vinfty}. 
 For each $b$ such that $w_b >0$, define $d_b$ by
\begin{equation}
\label{eq:db}
\lambda_{d_b}(\mathbb{S}^2,\widetilde V^b_\infty)<\Lambda_k(M,\mathcal C)\leqslant \lambda_{d_b+1}(\mathbb{S}^2,\widetilde V^b_\infty).
\end{equation}
Finally,  for any  type II bubble $a$ let  $w^m_a=\int_{\mathbb{S}^2} \widetilde{V}^a_m dv_{g_{\mathbb{S}^2}}$, where $\widetilde{V}^a_m$ is defined by \eqref{type2pot}.
Set $w_a=\lim_{m\to\infty} w_a^m$; note that the limit exists due to \eqref{A12}. For any type II bubble $a$ such that  $w_a>0$ ,   
define $d_a$ by 
\begin{equation}
\label{eq:da}
\limsup_{m\to \infty} \lambda_{d_a}(\mathbb{S}^2,\widetilde V^a_m)<\Lambda_k(M,\mathcal C)\leqslant \limsup_{m\to\infty} \lambda_{d_a+1}(\mathbb{S}^2,\widetilde V^a_m).
\end{equation}
Let $t$ be the number of necks of non-zero mass and terminal bubbles, and recall that we have assumed that the total area of the surface $M$ is equal to one.
\begin{proposition}
\label{prop:kest}
One has 
$$
(d_M+1) + \sum_{b\colon w^b\ne 0} (d_b+1) +\sum_{a\colon w^a\ne 0} (d_a+1) +  t \leqslant k. 
$$
In particular, $t\leqslant k$.
\end{proposition}
\begin{proof}
We argue by contradiction. Assume that 
$$
(d_M+1) + \sum_{b\colon w^b\ne 0} (d_b+1) + \sum_{a\colon w^a\ne 0} (d_b+1)  + t\geqslant k+1. 
$$
Then by Proposition~\ref{separation:proposition} one has 
$$
\lambda_k(M,V_{N_m,k})\leqslant \max\{\lambda_{d_M}(M,V_\infty), \lambda_{d_b}(\mathbb{S}^2, \widetilde V^b_\infty),\lambda_{d_a}(\mathbb{S}^2, \widetilde V^a_m)\} + o(1).
$$
Passing to the $\limsup$ as $m\to \infty$ and using the definition of $d_M,d_a,d_b$ one has
$$
\Lambda_k(M,\mathcal C)\leqslant \max\{\lambda_{d_M}(M,V_\infty), \lambda_{d_b}(\mathbb{S}^2, \widetilde V^b_\infty),\limsup\lambda_{d_a}(\mathbb{S}^2, \widetilde V^a_m)\} <\Lambda_k(M,\mathcal C).
$$
\end{proof}
We can now complete the proof of the main result of the paper.

\smallskip

\noindent{\it Proof of Theorem \ref{thm:main}.}
Let, as before, $M$ be a surface with a fixed conformal class $\mathcal{C}$.
Using the fact that $t\leqslant k$ one has
$$
w_M + \sum\limits_{\text{$b$ is not terminal}} w_b +  \sum\limits_{\text{$a$ is not terminal}} w_a   \geqslant 1-kC_R.
$$
Thus, summing up the inequalities 
\begin{equation*}
\begin{split}
w_M\Lambda_k(M,\mathcal C)&\leqslant w_M\lambda_{d_M+1}(M,V_\infty)\leqslant \Lambda_{d_M+1}(M,\mathcal C)\\
w_b\Lambda_k(M,\mathcal C)&\leqslant w_b\lambda_{d_b+1}(\mathbb{S}^2,\widetilde V^b_\infty)\leqslant \Lambda_{d_b+1}(\mathbb{S}^2)\\
w_a\Lambda_k(M,\mathcal C)&\leqslant w_a\limsup\lambda_{d_a+1}(\mathbb{S}^2,\widetilde V^a_m)\leqslant \Lambda_{d_a+1}(\mathbb{S}^2).
\end{split}
\end{equation*}
 yields, provided $w_M>0$,
\begin{multline*}
(1-kC_R)\Lambda_k(M,\mathcal C)\leqslant \Lambda_{d_M+1}(M,\mathcal C)+\sum_{b\colon w_b\ne 0} \Lambda_{d_b+1}(\mathbb{S}^2)  )+\sum_{a\colon w_b\ne 0} \Lambda_{d_a+1}(\mathbb{S}^2)\\
\leqslant \max_{k'}\{\Lambda_{k'}(M) + \Lambda_{k-k'}(\mathbb{S}^2)\},
\end{multline*}
where $1\le k'<k$ if there is at least one bubble of non-zero mass. Since the choice of $C_R$ is arbitrary, passing to the limit $C_R\to 0$ we obtain:
\begin{equation}
\label{eq:last}
\Lambda_k(M)\leqslant \max_{1\le k'<k}\{\Lambda_{k'}(M) + \Lambda_{k-k'}(\mathbb{S}^2)\}.
\end{equation}
At the same time, it was shown in \cite{KNPP} that $\Lambda_j (\mathbb{S}^2) =8\pi j$ for any $j~\ge~1$. 
Thus, if 
$$
\Lambda_k(M,\mathcal C)> \max_{k'<k}\{\Lambda_{k'}(M,\mathcal C) + 8\pi (k-k')\} = \Lambda_{k-1}(M,\mathcal C) + 8\pi
$$
then there are no bubbles and hence there exists a metric $h$ smooth outside of isolated conical singularities such that  $\Lambda_k(M) = \bar\lambda_k(M,h)$.  This proves the second assertion of Theorem \ref{thm:main}
provided $w_M>0$. If $w_M=0$ then instead of \eqref{eq:last} we would get
\begin{equation}
\label{lastlast}
\Lambda_k(M)\le 8\pi k.
\end{equation}
In view of \eqref{assum:sphere} and \cite[Theorem B]{CES}  it follows that $M$ is a sphere and \eqref{lastlast} is an equality. Moreover, it follows from the results of \cite{KNPP} that \eqref{mbubbles} holds.

Finally, for $k=1$,  Proposition~\ref{prop:kest} implies that only one of the weights $w_M,\,w_b,\,w_a$ could possibly be non-zero. If $w_M\ne 0$, then there are no bubbles and we obtain the  existence of a regular conformally maximal metric. If one of the bubbles has a non-zero mass, then by~\eqref{lastlast} one has
$$
\Lambda_1(M,\mathcal C)\leqslant 8\pi,
$$
which once again implies that $M$ is a sphere.
This completes the proof of Theorem~\ref{thm:main}.~\qed

\section{The Yang-Yau method  for higher eigenvalues}
\subsection{Spectra of $L^p$-measures}
\label{subsec:Lp}
In this section we collect some properties of the eigenvalues $\lambda_k(M,\mu)$ of $M$ with fixed conformal class $[g]$ and the measure $d\mu=\rho dv_g$ (see subsection  \ref{subsec:opt} for the setup), where $\rho\in L^p(M):=L^p(M,g)$ for some $p>1$. For the proof of the following proposition see~\cite[Propositions 2.13 and 2.14]{KS}, \cite[Section 2]{GKL}, as well as~\cite[Example 2.1]{Kok2}.
 \begin{proposition}
Suppose that  $\rho \in L^p(M)$  for some $p>1$ and $\rho \ge 0$. Then the spectrum of the Laplacian on $(M,\rho dv_g)$ is discrete, and the eigenvalues form a sequence
$$0=\lambda_0(M, \rho dv_g) < \lambda_1 (M, \rho dv_g) \le \lambda_2 (M,\rho dv_g) \le \dots \nearrow \infty.$$

The  eigenvalues $\lambda_k(M,\rho dv_g)$  have finite multiplicity, and  the corresponding eigenfunctions $\phi_k \in H^1(M):=H^1(M,g)$ satisfy the equation 
$$
\Delta_g\phi_k=\lambda_k(M,\rho dv_g)\rho\phi_k
$$
in the weak sense. 
\end{proposition}

 The following lemma appears to be known, but the authors were unable to find the exact reference. For similar results with slightly different formulations see~\cite[Proposition 3.14]{KS} and~\cite[Lemma 4.5]{CKM}. We include the proof below for completeness.
\begin{lemma}
\label{approximation:lemma}
Let $M$ be a surface endowed with the metric $g$, and  $[g]$ be the corresponding conformal class. Let $\rho_n$ be a sequence of non-negative functions such that $\rho_n\to \rho$ in $L^p(M)$ for some $p>1$. Then 
$$
\lim_{n\to\infty}\lambda_k(M,\rho_ndv_g)=\lambda_k(M,\rho dv_g), \,\,\, k=0,1,2,\dots
$$
 
\end{lemma}
\begin{proof}
The statement of the lemma is trivial for $k=0$ since the corresponding eigenvalues are equal to zero. 
By the upper semi-continuity of eigenvalues (see  Proposition \ref{usc} or \cite[Proposition 1.1]{Kok2}) it is sufficient to show that $$\liminf_{n\to\infty}\lambda_k(M,\rho_ndv_g)\geqslant\lambda_k(M,\rho dv_g), \,\,\, k\ge 1.$$
%
%

Replace $\{n\}$ with a subsequence $\{n_m\}$ so that 
$$
\lim_{m\to\infty}\lambda_i(M,\rho_{n_m}dv_g)=\liminf_{n\to \infty} \lambda_i(M,\rho_n dv_g)
$$
 for all $i\leqslant k$. To simplify notation, we rename the subsequence back to $\{n\}$.  Let $E_{k,n}$ be the space spanned by $\lambda_i(M,\rho_n dv_g)$-eigenfunctions for $i=1,\ldots, k$. 

Let $f_n\in E_{k,n}$ be normalized so that $\int f^2_n\rho_n\,dv_g = 1$. In particular one has
\begin{equation}
\label{eq:gradnorm}
\int\limits_{M} |\nabla f_n|^2\,dv_g\leqslant \lambda_k(M,\rho_n dv_g)\leqslant \lambda_k(M,\rho dv_g) + o(1),
\end{equation}
i.e. the Dirichlet integrals of $f_n$ are uniformly bounded.
Furthermore, we claim that 
\begin{equation}
\label{eq:boundedH1}
\|f_n\|_{H^{1}(M)}\leqslant C_k.
\end{equation}
In order to show this we recall the following theorem.
\begin{theorem}[\cite{AH}, Lemma 8.3.1]
\label{Poincare:thm}
Let $(M,g)$ be a Riemannian manifold. Then there exists a constant $C>0$ such that for all $L\in H^{-1}(M)$ with $L(1) = 1$ one has 
\begin{equation}
\|u- L(u)\|_{L^2(M)} \leqslant C\|L\|_{H^{-1}(M)} \left(\,\int\limits_M |\nabla u|^2_g\,dv_g\right)^{1/2}
\end{equation}
for all $u\in H^1(M)$.
\end{theorem}

We apply Theorem \ref{Poincare:thm} to  $L_n(u) = \int_M u\rho_n\,dv_g$.  Let $q$ be the  H\"older dual of $p$. Then
$$
\int_M u\rho_n\leqslant \|\rho_n\|_{L^p(M)} \|u\|_{L^{q}(M)}\leqslant C\|\rho\|_{L^p}\|u\|_{H^{1}(M)}.
$$
The last inequality follows from the Rellich-Kondrachov theorem (see, for instance, \cite[Theorem 1.1]{Kaz}), stating that the embedding $H^{1}(M)\subset L^q(M)$ is compact for any $1\leqslant q < \infty$.
Therefore, $\|L_n\|_{H^{-1}(M)}$ are uniformly bounded.
Theorem~\ref{Poincare:thm} then yields
$$
\int_M \left(f_n - \int f_n\rho_n\,dv_g\right)^2\,dv_g = \int_M f_n^2\,dv_g\leqslant C\int_M |\nabla f_n|^2\leqslant C_k,
$$
where the first equality follows from the fact that eigenfunctions are orthogonal to constants.
Together with \eqref{eq:gradnorm} this implies \eqref{eq:boundedH1}.
As a consequence, by Rellich-Kondrachov theorem we get  $\|f_n\|_{L^q(M)}\leqslant C_{k,q}$ for any $1 \leqslant q < \infty$.

Let $\phi_{n,i}\in E_{k,n}$ be a normalized basis of eigenfunctions so that $$\int_M \phi_{n,i}\phi_{n,j}\rho_n = \delta_{ij}.$$ 
Then, by \eqref{eq:boundedH1}, up to a choice of a subsequence,   \{$\phi_{n,i}$\}  converges  as $n\to \infty$ weakly in $H^1(M)$ and strongly in $L^{2q}(M)$. Here the first assertion follows
from the Banach-Alaoglu theorem, and the second one from the compactness of the embedding  $H^{1}(M)\subset L^q(M)$.
%
Let $\phi_i\in H^1(M)$ be the corresponding limits. We claim that $\{\phi_i\}$ is a normalized collection of eigenfunctions for the measure $\rho$, and the values $\lambda_i = \lim_{n\to\infty} \lambda_i(M,\rho_n dv_g)$ are the corresponding eigenvalues.

Indeed, since $\phi_{i,n}\to \phi_i$ in $L^{2q}(M)$, then $\phi_{i,n}\phi_{j,n}\to\phi_i\phi_j$ in $L^q(M)$. Therefore,
\begin{equation*}
\begin{split}
&\left|\int (\phi_i\phi_j\rho-\phi_{i,n}\phi_{j,n}\rho_n)\,dv_g\right| \leqslant \\
& \|\phi_{i,n}\phi_{j,n}-\phi_i\phi_j\|_{L^q}\|\rho_n\|_{L^p} + \|\phi_{i,n}\phi_{j,n}\|_{L^q}\|\rho-\rho_n\|_{L^p}\to 0,
\end{split}
\end{equation*}
i.e. the functions $\phi_i$ are normalized so that $\int \phi_{i}\phi_j\rho\,dv_g = \delta_{ij}$. In particular, $\phi_i$ are linearly independent.

Finally, we show that $\phi_i$ are  (weak) eigenfunctions for the measure $\rho$ with the corresponding eigenvalues $\lambda_i$. 
Indeed, given $\psi\in C^\infty (M)$, we obtain
\begin{equation*}
\begin{split}
&\int \nabla\phi_i\nabla\psi\,dv_g = \lim_{n\to\infty}\int\nabla \phi_{i,n}\nabla\psi\,dv_g = \\
&\lim_{n\to \infty} \lambda_i(M,\rho_n\,dv_g)\int\phi_{i,n}\psi\rho_n\,dv_g = \lambda_i\int\phi_i\psi\rho\,dv_g.
\end{split}
\end{equation*}
Note that a priori  we do not claim that $\lambda_i$ is necessarily the $i$-th eigenvalue for the measure $\rho$, but simply that it is some eigenvalue  $\lambda_i\geqslant \lambda_i(M,\rho dv_g)$; however, the equality in fact holds by the upper-semicontinuity property mentioned earlier.
This completes the proof of Lemma \ref{approximation:lemma}.
\end{proof}

\begin{corollary}
\label{cor:Lpconv}
Suppose that  $\rho \in L^p(M,g)$  for some $p>1$ and $\rho \ge 0$. Then
$$
\lambda_k(M,\rho dv_g)\int_M\rho\,dv_g\leqslant \Lambda_k(M,[g]).
$$
\end{corollary}
\begin{proof}
Any non-negative $\rho\in L^p(M,g)$ can be approximated in $L^p(M,g)$ by smooth positive functions $\rho_n$. For such $\rho_n$ one has $g_n =\rho_n g\in [g]$ and 
$$
\lambda_k(M,\rho_n dv_g)\int_M\rho_n\,dv_g = \lambda_k(M,g_n)\area(M,g_n).
$$ 
An application of Lemma~\ref{approximation:lemma} completes the proof.
\end{proof}

\subsection{Proof of Theorem~\ref{thm: Korevaar}}
\label{sec:Korevaar}

We  prove part (i) first. Let $g$ be a Riemannian metric on $M$ and  $[g]$ be the corresponding conformal class.
Following \cite{YY}, let $\pi\colon M \to \mathbb{S}^2$ be a conformal branched covering of degree $d$, where $\mathbb{S}^2$ is endowed with standard metric $g_0$ on a unit sphere and the corresponding conformal structure.
Consider the  push-forward $d\mu = \pi_*dv_g$ of the volume measure on $M$. By ~\cite[equation (2.4)]{YY} the measure $d\mu$ satisfies $d\mu= \rho dv_{g_0}$, where $\rho \in L^p(\mathbb{S}^2,g_0)$ for some $p>1$. 
\begin{remark}
The local expression for $d\mu$ obtained in~\cite[equation (2.4)]{YY} implies that $d\mu=dv_{g^*}$, where $g^*= \rho g_0$ is a metric on 
$\mathbb{S}^2$ with conical singularities at images of branch points. Note, however, that the conical angles at these singularities are smaller than $2\pi$, which forces the conformal factor to be unbounded around the singularity.  
\end{remark}

Consider the $(k+1)$-dimensional subspace 
$E_{k+1}^* \in H^1(\mathbb{S}^2,g_0)$ spanned by the eigenfunctions corresponding to the eigenvalues $0, \lambda_1(\mathbb{S}^2, \mu), \dots, \lambda_k(\mathbb{S}^2, \mu)$.
Consider  now the space $E_{k+1} \in H^1(M,g)$ consisting of functions $u=u^* \circ \pi$, $u^* \in E_{k+1}^*$. Then,  by the variational principle, 
\begin{multline}
\label{eq:Korexpl}
\lambda_k(M,g)  \le \sup_{u \in E_{k+1}} \frac{\int_M |\nabla u|^2 dv_g}{\int_M u^2 dv_g}=\\  d \cdot  \sup_{u \in E_{k+1}^*} \frac{\int_{\mathbb{S}^2} |\nabla u^*|^2 dv_{g_0}}{\int_{\mathbb{S}^2} (u^*)^2 d\mu}=d \cdot  \lambda_k (\mathbb{S}^2, \mu) \le \frac{8\pi k d}{\mu(\mathbb{S}^2 )}.
\end{multline}
Here the first inequality follows from the variational principle, the last inequality is true by \eqref{eq:KNPP} and Corollary \ref{cor:Lpconv},  and the equality in the middle follows from \cite[Lemma, p. 59]{YY}.
Setting $u=1$ in part (i) of the same Lemma in \cite{YY} we note that $\area(M,g)=\mu(\mathbb{S}^2)$. Finally, as was shown in 
\cite{EI},  we can set $d= \left[\frac{\gamma+3}{2}\right]$.  Therefore, \eqref{eq:Korexpl} implies \eqref{eq:KorexplO}.

In order to prove part (ii), we argue in the same way, using the method of \cite{Kar0} instead of \cite{YY}. The key observation is that an analogue of  \cite[Proposition 1]{Kar0} holds
for $\Lambda_k$ with  the factor $8\pi$ on the right replaced by $8\pi k$, and the factor  $12 \pi$ replaced by $4\pi(2k+1)$, where the former is true by  \eqref{eq:KNPP}, and the latter  follows from \eqref{eq:projective}. We leave the rest of the details to the reader.~\qed

\begin{remark} In view of \cite[Lemma 3]{Kar0}, if $M$ is  the Dyck's surface (i.e. a non-orientable surface with an  orientable double cover of genus two), inequality \eqref{eq:KorexplN} can be improved to
$\Lambda_k (M) \le 16 \pi k$, $k\ge 1$.
\end{remark}

\end{document}